\documentclass[11pt]{amsart}
\usepackage{amsmath,amsfonts,amsthm}
\usepackage{amssymb,latexsym}
\usepackage[colorlinks=true, linkcolor=blue,urlcolor=blue]{hyperref}
\usepackage{graphicx}
\usepackage[all]{xy}
\usepackage{enumitem}

\theoremstyle{plain}
\newtheorem{theorem}{Theorem}[section]
\newtheorem{lemma}[theorem]{Lemma}
\newtheorem{proposition}[theorem]{Proposition}
\newtheorem{corollary}[theorem]{Corollary}

\newtheorem{thmx}{Theorem}

\newtheorem{corx}[thmx]{Corollary}

\theoremstyle{definition}

\newtheorem{definition}[theorem]{Definition}
\newtheorem{example}[theorem]{Example}

\newtheorem{remark}[theorem]{Remark}

\newtheorem{notation}[theorem]{Notation}

\theoremstyle{remark}
\numberwithin{equation}{section}

\newcounter{TmpEnumi}

\newcommand{\N}{\mathbb{N}}

\newcommand{\Ped}{{\operatorname{Ped}}}

\newcommand{\af}{\alpha}
\newcommand{\bt}{\beta}

\newcommand{\dt}{\delta}
\newcommand{\ep}{\varepsilon}

\newcommand{\et}{\eta}

\newcommand{\ld}{\lambda}
\newcommand{\sm}{\sigma}

\newcommand{\ph}{\varphi}
\newcommand{\ps}{\psi}
\newcommand{\rh}{\rho}

\newcommand{\ta}{\tau}

\newcommand{\card}{{\operatorname{card}}}
\newcommand{\diag}{{\operatorname{diag}}}
\newcommand{\Aut}{{\operatorname{Aut}}}

\newcommand{\andeqn}{\qquad {\mbox{and}} \qquad}

\begin{document}

\title[Actions on  tracially $\mathcal{Z}$-absorbing C*-algebras]{Group actions on simple tracially $\mathcal{Z}$-absorbing C*-algebras}



\author[M. Amini, N. Golestani,  S. Jamali, N. C. Phillips]{
  Massoud Amini, Nasser Golestani, Saeid Jamali,
  N. Christopher Phillips}


\address{{\textbf{M. Amini, N. Golestani, S. Jamali}} \newline
           Faculty of Mathematical Sciences, Tarbiat Modares University
\\
           Tehran 14115-134
\newline
              {\textbf{N. C. Phillips}} \newline
   Department of Mathematics, University of Oregon
\\
   Eugene OR 97403-1222}
\thanks{This research of the second author was
 in part supported by a grant from IPM (No.1401460118)}



\begin{abstract}
We show that if $A$ is a simple
(not necessarily unital) tracially $\mathcal{Z}$-absorbing C*-algebra
and  $\alpha \colon G \to \mathrm{Aut} (A)$ is an
action of a finite group $G$ on $A$  with the weak tracial Rokhlin
property, then
the crossed product $C^*(G, A,\alpha)$ and the fixed point
algebra $A^\alpha$ are simple and tracially $\mathcal{Z}$-absorbing,
and they are $\mathcal{Z}$-stable if, in addition, $A$ is 
separable and nuclear. The same conclusion
holds for all intermediate C*-algebras of the inclusions
$A^\alpha \subseteq A$ and $A \subseteq C^*(G, A,\alpha)$.
We  prove that if $A$
is a simple tracially $\mathcal{Z}$-absorbing C*-algebra,
then, under a finiteness condition,
the permutation action of the symmetric group $S_m$
on the minimal $m$-fold tensor product of $A$
has the weak tracial Rokhlin property.
We define the weak tracial Rokhlin property for automorphisms
of simple  C*-algebras
and we show that---under
a mild assumption---(tracial) $\mathcal{Z}$-absorption is
preserved under crossed products by such automorphisms.
\end{abstract}

\maketitle
\tableofcontents
\section{Introduction and main results}\label{sec_intro}
%



The study of group actions on C*-algebras  and 
the corresponding crossed products is at the core of 
 the modern theory of operator algebras.
With the near completion of the Elliott's program
to classify simple separable unital nuclear $\mathcal{Z}$-stable
C*-algebras   (see W.~Winter's ICM talk in 2018 
\cite{Winter2018} for details),
the classification of nonunital C*-algebras is now of great
interest \cite{GngLin, GngLinII, GngLinIII, EGLN17}. Among equivalent conditions which characterize classifiability, the
notion of $\mathcal{Z}$-stability has a central role in the classification program.
The  property of tracial $\mathcal{Z}$-absorption, studied by Hirshberg and
Orovitz for unital C*-algebras in \cite{HO13}, and by the authors 
for nonunital C*-algebras in \cite{AGJP21}, is a 
local version of $\mathcal{Z}$-stability, and is
equivalent to that for simple separable nuclear
C*-algebras \cite{HO13, CLS2021}.

\medskip

In this paper we investigate finite group actions on
simple (not necessarily unital) tracially 
$\mathcal{Z}$-absorbing C*-algebras.
In particular, we deal with the $\mathcal{Z}$-stability of 
the resulting crossed products. For actions on 
simple \emph{unital} C*-algebras, there are already nice results
in the literature
(however, even in the unital case, some of our results 
are not covered by the existing known results).
For instance, Matui and Sato 
showed that if 
$\alpha \colon \Gamma \to \mathrm{Aut} (A)$
is a strongly outer action of an elementary amenable group $\Gamma$
on a  simple unital separable  nuclear stably finite
infinite dimensional C*-algebra $A$ with finitely many
extremal tracial states, then the $\mathcal{Z}$-stability
of $A$ implies that of the crossed product
$C^* (\Gamma, A,\alpha)$ (see \cite[Corollary~4.11]{MS2014}).
In \cite[Theorem~5.6]{HO13}, Hirshberg and Orovitz
showed that if  $\alpha \colon G \to \mathrm{Aut} (A)$ is an
action of a finite group $G$ on a simple unital separable
tracially $\mathcal{Z}$-absorbing
C*-algebra $A$, then
the crossed product $C^*(G, A,\alpha)$ is tracially
$\mathcal{Z}$-absorbing, whenever 
$\alpha$ has  the generalized tracial Rokhlin
property. 

\medskip

There are a few results about the $\mathcal{Z}$-stability
of the crossed products of actions on \emph{nonunital}
C*-algebras, but the assumptions on the actions
are rather strong. For instance, if
 $\alpha \colon G \to \mathrm{Aut} (A)$ is an action 
of a finite group $G$ on a separable  C*-algebra $A$,
then the $\mathcal{Z}$-stability
of $A$ implies that of the crossed product, 
provided that $\alpha$ has the Rokhlin property 
\cite[Theorem~3(v)]{San15}. More generally, 
the same conclusion holds even if $\alpha$
has finite Rokhlin dimension with commuting towers 
\cite[Theorem~2.2]{HrsPh1}.
However, the Rokhlin property is rather strong  and imposes
certain restrictions on K-theory of the C*-algebra.
For instance, if $A$ is simple and separable
such that  either $K_0(A)$
or $K_1(A)$ is isomorphic to $\mathbb{Z}$, then there 
is no nontrivial finite group actions
on $A$ with the Rokhlin property 
(see \cite[Corollary~3.10]{Naw2016}). Likewise,
there is no action of a nontrivial finite group  on 
$\mathcal{Z}$ or $\mathcal{O}_\infty$ which has finite 
Rokhlin dimension with commuting towers \cite[Corollary~4.8]{HrsPh1}
(cf.~Theorem~\ref{thmxB} and Corollary~\ref{cor_per_any} below).
For this reason, one would prefer 
instead to use the weak tracial Rokhlin property \cite{FG17}.

In the first main result we deal with 
 the preservation of (tracial) $\mathcal{Z}$-absorption 
for finite groups actions on nonunital simple C*-algebras.  

\begin{thmx}\label{thmxA}
Let $A$ be a simple tracially $\mathcal{Z}$-absorbing C*-algebra
and let $\alpha \colon G \to \mathrm{Aut} (A)$ be an
action of a finite group $G$ on $A$  with the weak tracial Rokhlin
property. Then
\begin{enumerate}
\item\label{thmxA_it1}
the crossed product $C^*(G, A,\alpha)$ and the fixed point
algebra $A^\alpha$ are simple and tracially $\mathcal{Z}$-absorbing;
\item\label{thmxA_it2}
if $A$ is $\sigma$-unital 
then all 
intermediate 
C*-algebras $B$ and $D$ with 
\[
A^{\alpha}  \subseteq B \subseteq A \subseteq D \subseteq C^*(G, A,\alpha)
\]
are   simple and
tracially $\mathcal{Z}$-absorbing;

\item\label{thmxA_it3}
if $A$ is separable and nuclear then all intermediate
C*-algebras $B$ and $D$ as above are nuclear and $\mathcal{Z}$-stable.
\end{enumerate}
\end{thmx}

Part~\eqref{thmxA_it1} extends \cite[Theorem~5.6]{HO13} to
the nonunital case in which  the fixed point algebra is not considered.
 (Note that, the assumption of separability
is implicit in the statement of \cite[Theorem~5.6]{HO13}, since the 
authors use their Lemma~5.5 in the proof.) 
To prove Part~\eqref{thmxA_it1}, we use ideas from \cite{HO13} 
and the machinery of
 Cuntz subequivalence in our nonunital setting
instead of using the tracial state space (see Theorem~\ref{thm_fg}).

  Part~\eqref{thmxA_it2} provides a way to identify more tracially 
  $\mathcal{Z}$-absorbing C*-algebras, by looking at intermediate C*-algebras.
  Recently, the intermediate C*-algebras
  (and their classification) have attracted some interest \cite{ER2021, AK2020}.
  We use Izumi's result on  the Galois correspondence for
  intermediate C*-algebras in the setting of finite group actions
  \cite[Corollary~6.6]{Izumi2002} 
  (see Corollary~\ref{Cor_intermediate} below). 
  Part~\eqref{thmxA_it3} follows from Part~\eqref{thmxA_it2}
  and the equivalence of $\mathcal{Z}$-stability and tracial $\mathcal{Z}$-absorption
 for simple separable nuclear C*-algebras \cite{HO13, CLS2021}.
 See Corollary~\ref{cor_fg_zs} for the statement and proof of this part.
 
 \medskip
 
 In \cite{HO13}, it is shown that the permutation action (see Section~\ref{Sec_Perm})
 of the symmetric group $S_m$ on  
 $\mathcal{Z}^{\otimes m}\cong \mathcal{Z}$
 has the generalized tracial Rokhlin property
 (and hence the weak tracial Rokhlin property, by Proposition~\ref{P_8615_WeakToGen}).
 Our next result is about the weak tracial 
 Rokhlin property of the permutation action
 of the symmetric group $S_m$ on the
  minimal tensor product of $m$ copies of
  a simple tracially $\mathcal{Z}$-absorbing C*-algebra.
  The combination of this result and Theorem~\ref{thmxA} 
  could be used to give new examples $\mathcal{Z}$-stable C*-algebras.
  
\begin{thmx}\label{thmxB}
Let $A$
be a simple tracially $\mathcal{Z}$-absorbing C*-algebra and 
$A^{\otimes m}$ be finite for some $m \in \N$.
Then the permutation action $\bt \colon S_m \to \Aut (A^{\otimes m})$
has the weak tracial Rokhlin property.
\end{thmx}

Here, finiteness of $A^{\otimes m}$ means that its unitization
is finite in the usual sense, which is the case,
for example, if $A^{\otimes m}$ is stably projectionless
or if $A$ is simple, exact, and stably finite
in the sense of \cite{Ro02}.

Though the idea of the proof is rather simple,
the proof is  long and needs several technical lemmas
(see  Section~\ref{Sec_Perm}).
Also, we use a recent result  which says
that for actions on finite C*-algebras, Condition~\eqref{WTRP-OUT-CO}
in the definition of the weak tracial Rokhlin property
(Definition~\ref{defwtrp}) follows from 
the other conditions \cite[Proposition~7.12]{AGJP21}.

 \medskip

As an application of Theorems~\ref{thmxB}, the permutation
action of $S_m$ on the Razak-Jacelon algebra 
$\mathcal{W}\cong \mathcal{W}^{\otimes m}$ has the 
weak tracial Rokhlin property, since $\mathcal{W}$ is simple, finite, and 
$\mathcal{Z}$-stable. Also, Theorem~\ref{thmxA} implies that
$C^*(S_m, \mathcal{W}, \beta)$ is $\mathcal{Z}$-stable (since 
$\mathcal{W}$  is nuclear). In \cite[Example~3.12]{FG17}, 
various examples of actions
with the weak tracial Rokhlin property are constructed 
using Theorem~\ref{thmxB}.
For instance, if $A$ is a simple nonelementary C*-algebra with
tracial rank zero
(hence  tracially $\mathcal{Z}$-absorbing, by \cite[Theorem~A]{AGJP21}), 
then $\bt \colon S_m \to \Aut (A^{\otimes m})$
has the weak tracial Rokhlin property and the resulting
crossed product is  tracially $\mathcal{Z}$-absorbing, and if, in addition,
$A$ is separable and nuclear, then the crossed product is   $\mathcal{Z}$-stable.

\medskip
Using a recent dichotomy for tracially approximately divisible C*-algebras \cite{FLL21}
and Theorems~\ref{thmxA} and \ref{thmxB} we obtain the following.
See Corollary~\ref{cor_per_tz}.

\begin{corx}\label{corxC}
Let $A$ be a simple separable exact tracially $\mathcal{Z}$-absorbing
C*-algebra. Let $m\in\mathbb{N}$ and consider
 the permutation action $\bt \colon S_m \to \Aut (A^{\otimes m})$.
 Then  all intermediate C*-algebras of the inclusions
$(A^{\otimes m})^\beta \subseteq A^{\otimes m}$ and 
$A^{\otimes m} \subseteq C^*(S_m, A^{\otimes m},\beta)$
are simple and tracially $\mathcal{Z}$-absorbing. 
If, in addition, $A$ is nuclear then 
all these intermediate C*-algebras 
 are   nuclear and $\mathcal{Z}$-stable.
\end{corx}

\medskip
The next main result is about $\mathbb{Z}$-actions
and (tracial) $\mathcal{Z}$-absorption of the crossed 
product. This result is a nonunital version
of \cite[Theorem 6.7]{HO13}.
We weaken
the assumption in \cite[Theorem 6.7]{HO13}
that $\alpha^{m}$ acts trivially on $T (A)$ for some
$m \in \mathbb{N}$,
and we don't need separability.
We extend the definition of the (weak) tracial Rokhlin property
for actions of
$\mathbb{Z}$ to the nonunital case (cf.~\cite[Definition~1.1]{OP06}).
See Definitions~\ref{defwtrpaut} and \ref{defwtrpaut_cont}
and Theorem~\ref{thm_int}.

\begin{thmx}\label{thmxD}
Let $A$ be a simple tracially $\mathcal{Z}$-absorbing C*-algebra and let
$\alpha \in \mathrm{Aut} (A)$ have the controlled weak tracial
Rokhlin property.
Then the crossed product $C^{*} (\mathbb{Z}, A, \alpha)$ is    simple and
tracially $\mathcal{Z}$-absorbing.
If moreover $A$ is separable and nuclear
 then   $C^{*} (\mathbb{Z}, A, \alpha)$ is $\mathcal{Z}$-stable.
\end{thmx}

As an example, if $A$ is any simple $\mathcal{Z}$-stable C*-algebra,
 then, starting from an arbitrary automorphism  of $A$,
we obtain an automorphism $\alpha \in \mathrm{Aut} (A)$
with the controlled weak tracial 
Rokhlin property (Example~\ref{Ex_auto}). The resulting 
crossed product is tracially $\mathcal{Z}$-absorbing,
and it is  $\mathcal{Z}$-stable if, in addition, $A$ is separable and nuclear.

\medskip

Note that 
the \emph{controlled} weak tracial
Rokhlin property is equivalent to
the weak tracial
Rokhlin property for automorphisms of 
simple unital exact
tracially $\mathcal{Z}$-absorbing C*-algebras
with finitely many extremal traces
(see Corollary~\ref{cor_cwtrp}).
\medskip

The  paper is organized as follows.
In Section~\ref{sec_crossf},
we prove Theorem~\ref{thmxA}.
In Section ~\ref{Sec_Perm},
we prove that if $A$
is a simple tracially $\mathcal{Z}$-absorbing C*-algebra,
then, under a finiteness condition,
the permutation action of the symmetric group
on the minimal tensor product of finitely many copies of~$A$
has the weak tracial Rokhlin property.
In Section~\ref{sec_int}, we first
define the weak tracial Rokhlin property for automorphisms
of simple (not necessarily unital) C*-algebras.
Then we show---under a mild assumption---that (nonunital)
tracial $\mathcal{Z}$-absorption passes to
crossed products by automorphisms
with the weak tracial Rokhlin property (Theorem~\ref{thm_int}).

\medskip

We use the following (standard) notations.
For a C*-algebra $A$, $A_{+}$ denotes the positive cone of $A$.
Also, $A^{+}$ denotes the minimal unitization of $A$
(adding a new identity even if $A$ is unital),
while $A^{\sim} = A$ if $A$ is unital and
$A^{\sim} = A^{+}$ if $A$ is nonunital.
The notation $a \approx_{\varepsilon} b$
means $\|a - b\| < \varepsilon$.
For $a,b\in A_+$, we use $a\precsim b$ to mean
that $a$ is Cuntz subequivent to $b$.
(See, for example, \cite{Ph14, AGJP21} for 
 preliminaries on the Cuntz subequivalence.)
We write $\mathcal{K} =\mathbb K (\ell^{2})$ and
${M}_{n} =\mathbb  M_{n} (\mathbb{C})$.
We take ${\mathbb{Z}}_n = {\mathbb{Z}} / n {\mathbb{Z}}$.
(The $p$-adic integers will never appear.)
We take $\mathbb{N} = \{ 1, 2, \ldots \}$.
We abbreviate ``completely positive contractive'' to ``c.p.c.''.
If $\af \colon G \to {\mathrm{Aut}} (A)$ is an action of a group~$G$
on a C*-algebra~$A$,
then we denote  the fixed point algebra by $A^{\af}$.
Finally, for any C*-algebra~$A$,
we denote its tracial state space by $T (A)$
and by $A^{\otimes m}$ the minimal tensor product of
$m$ copies of $A$.


\indent

\section{Finite group actions}\label{sec_crossf}

\indent
For actions of finite groups on not necessarily unital C*-algebras,
at least five Rokhlin type properties have appeared in the literature.
One is the multiplier Rokhlin property \cite[Definition~2.15]{Ph09},
which is defined using projections in the multiplier algebra.
A second, given in \cite[Definition~3.1]{Naw2016}
and just called the Rokhlin property,
is defined for $\sigma$-unital C*-algebras using projections
in the central sequence algebra of the given C*-algebra.
A third,
Definition~2 in \cite[Section~3]{San15},
is called there just the Rokhlin property,
and is defined using positive contractions
in the algebra instead of projections.
In the unital case,
it is equivalent to the usual Rokhlin property
(Corollary~1 in \cite[Section~3]{San15}),
in the separable nonunital case
it is equivalent to \cite[Definition~3.1]{Naw2016}
(Corollary~2 in \cite[Section~3]{San15}),
and it is implied by the multiplier Rokhlin property
(Corollary~3 in \cite[Section~3]{San15}),
although the converse is false
(Example~1 in \cite[Section~3]{San15}).
A nonunital version of finite Rokhlin dimension with commuting towers
is given in \cite[Definition 1.14]{HrsPh1},
and of course there is an analog without commuting towers.
The fifth Rokhlin type property
is a weak version
(using positive elements instead of projections)
of the tracial Rokhlin property
for finite group actions on simple
C*-algebras \cite{FG17}.
This is the one we use here.
It is recalled in Definition~\ref{defwtrp} below.
Even in the unital case,
it is much weaker than the first three
properties above,
as discussed in \cite[Example 3.12]{Ph09},
and
it also does not imply finite Rokhlin dimension with commuting towers,
as follows from \cite[Corollary 4.8]{HrsPh1}.
It seems likely that finite Rokhlin dimension with commuting towers
implies the property in Definition~\ref{defwtrp}.
As far as we know,
however, this has never been checked.

In this section we prove
(Theorem~\ref{thm_fg} below) that
if $\alpha \colon G \to \mathrm{Aut} (A)$ is an
action of a finite group $G$ on a simple (not necessarily unital)
tracially $\mathcal{Z}$-absorbing C*-algebra
and $\alpha$ has the weak tracial Rokhlin property,
then $C^{*} (G, A, \alpha)$ and $A^{\alpha}$ are also simple
tracially $\mathcal{Z}$-absorbing C*-algebras.
This is a generalization
of Theorem~5.6 of \cite{HO13} to the nonunital case.
In \cite{HO13} the fixed point algebra is not considered.
We also remove the assumption of separability
from Theorem~5.6 of \cite{HO13}.
(This assumption
is implicit in the statement of Theorem~5.6 of \cite{HO13}
because the authors use their Lemma~5.5 in the proof.)

First we recall 
the definition
of  tracial  $\mathcal{Z}$-absorption
introduced in  \cite{HO13} in the  unital case
and exteneded to the nonunital case in \cite{AGJP21}.

	%
		%
		%
	%

\begin{definition}[\cite{AGJP21}, Definition~3.6]\label{deftza}
	We say that a simple C*-algebra $A$
	is \emph{tracially $\mathcal{Z}$-absorbing} if
	$A \ncong \mathbb{C}$
	and for every $x, a \in A_{+}$ with $a \neq 0$,
	every finite set $F \subseteq A$, every $\varepsilon > 0$,
	and every $n \in \mathbb{N}$,
	there is a c.p.c.~order zero map $\varphi \colon M_{n} \to A$
	such that:
	\begin{enumerate}
		\item\label{deftza-it1}
		$\bigl( x^{2} - x \varphi (1) x - \varepsilon \bigr)_{+} \precsim a$.
		\item\label{deftza-it2}
		$\| [\varphi (z), b] \| < \varepsilon$
		for any $z \in M_{n}$ with $\| z \| \leq 1$ and any $b \in F$.
	\end{enumerate}
\end{definition}

We recall the definition of the
weak tracial Rokhlin property
for finite group actions on simple (not necessarily unital)
C*-algebras from \cite{FG17}.

\begin{definition}[\cite{FG17}, Definition~4.4]\label{defwtrp}
Let $\alpha \colon G \to \mathrm{Aut} (A)$
be an action of a finite group $G$ on a simple C*-algebra~$A$.
Then $\alpha$ has the {\emph{weak tracial Rokhlin property}} if
for every $\varepsilon > 0$,
every finite set $F \subseteq A$,
and every $x, y \in A_{+}$ with $\| x \| = 1$,
there exist orthogonal positive contractions
$f_g \in A$ for $g \in G$ such that,
with $f = \sum_{g \in G} f_{g}$,
the following hold:
\begin{enumerate}
\item\label{WTRP-COM}
$\| f_{g} a - a f_{g} \| < \varepsilon$
for all $a \in F$ and all $g \in G$.
\item\label{WTRP-LTR}
$\|\alpha_{g} (f_{h}) - f_{g h} \| < \varepsilon$ for all $g, h \in G$.
\item\label{WTRP-CU-EQ}
$\big( y^{2} - y f y - \varepsilon \big)_{+} \precsim x$.
\item\label{WTRP-OUT-CO}
$\|f x f\| > 1 - \varepsilon$.
\end{enumerate}
\end{definition}

We relate the weak tracial Rokhlin property
to the generalized tracial Rokhlin property
of \cite{HO13}.

\begin{proposition}\label{P_8615_WeakToGen}
Let $\alpha \colon G \to \mathrm{Aut} (A)$
be an action of a finite group $G$
on a simple unital C*-algebra~$A$.
\begin{enumerate}[label=$\mathrm{(\arabic*)}$]
\item\label{P_8615_WeakToGen_ToGen}
If $\alpha$ has the weak tracial Rokhlin property,
then $\alpha$ has the generalized tracial Rokhlin property
of \cite[Definition 5.2]{HO13}.
\item\label{P_8615_WeakToGen_ToWk}
If $A$ is finite
and $\alpha$ has the generalized tracial Rokhlin property
of \cite[Definition 5.2]{HO13},
then $\alpha$ has the weak tracial Rokhlin property.
\end{enumerate}
\end{proposition}


\begin{proof}[Proof of Proposition~\ref{P_8615_WeakToGen}]
Suppose $\alpha$ has the weak tracial Rokhlin property.

Define continuous functions
$k, l \colon [0, 1] \to [0, 1]$
by
\[
k ( t )
 = \begin{cases}
   3 t & \hspace*{1em} 0 \leq t \leq \frac{1}{3}
       \\
   1   & \hspace*{1em} \frac{1}{3} < t \leq 1
\end{cases}
\andeqn
l (t)
 = \begin{cases}
   0       & \hspace*{1em} 0 \leq t \leq \frac{2}{3}
       \\
   3 t - 2 & \hspace*{1em} \frac{2}{3} < t \leq 1.
\end{cases}
\]

To prove that $\alpha$ has the generalized tracial Rokhlin property,
let $\varepsilon > 0$,
let $F \subseteq A$ be finite,
and let $a \in A_{+} \setminus \{ 0 \}$.
Without loss of generality,
$\| c \| \leq 1$ for all $c \in F$.

Use \cite[Lemma 2.5.11(2)]{LnBook}
to choose $\dt_1 > 0$
such that whenever $b, c \in A$
satisfy $0 \leq b, c \leq 1$
and $\| b - c \| < \dt_1$,
then $\| k (b) - k (c) \| < \ep$.
Use \cite[Lemma 2.5]{AP16}
to choose $\dt_2 > 0$
such that whenever $b \in A$
satisfies $0 \leq b \leq 1$
and $c \in F$
satisfies $\| [b, c] \| < \dt_2$,
then $\| [k (b), \, c] \| < \ep$.

Define
$\varepsilon_0
 = \min \big( \dt_1, \dt_2, \frac{1}{3} \big)$.
Set $x = \| a \|^{-1} a$ and $y = 1$.
Apply Definition~\ref{defwtrp}
with $\varepsilon_0$ in place of $\varepsilon$, with $F$ as given,
and these choices of $x$ and~$y$,
getting orthogonal positive contractions
$f_g \in A$ for $g \in G$.
Set $f = \sum_{g \in G} f_{g}$.
Condition~\eqref{WTRP-OUT-CO} in Definition~\ref{defwtrp}
and $\varepsilon_0 \leq \frac{1}{3}$ imply $\| f x f \| > \frac{2}{3}$.
In particular,
$\| f \| > \big( \frac{2}{3} \big)^{1/2} > \frac{2}{3}$.
So there is at least one $g_0 \in G$
such that $\| f_{g_0} \| > \frac{2}{3}$,
and now Condition~\eqref{WTRP-LTR} in Definition~\ref{defwtrp}
implies $\| f_{g} \| > \frac{1}{3}$ for all $g \in G$.

Define $e_g = k (f_g)$ for $g \in G$.
The elements $e_g$ are orthogonal positive contractions,
and the inequality $\| f_{g} \| > \frac{1}{3}$
ensures that $\| e_g \| = 1$ for all $g \in G$.
In particular, we have (1) in \cite[Definition 5.2]{HO13}.

By $\varepsilon_0 \leq \dt_1$ and the choice of $\dt_1$,
we have
$\| \alpha_{g} (e_{h}) - e_{g h} \| < \varepsilon$ for all $g, h \in G$,
which is (4) in \cite[Definition 5.2]{HO13}.
Using $\dt_2$ in place of~$\dt_1$,
we similarly get
$\| e_{g} a - a e_{g} \| < \varepsilon$
for all $a \in F$ and all $g \in G$,
which is (3) in \cite[Definition 5.2]{HO13}.

Define $e = \sum_{g \in G} e_{g}$.
It remains to show that
$1 - e \precsim a$.
One easily checks that
$e = k (f)$.
We have
$1 - k (t) = l (1 - t)$
for all $t \in [0, 1]$,
and for any $b \in A$ with $0 \leq b \leq 1$
we have
$l (b) \sim \big( b - \frac{2}{3} \big)_{+}$.
Therefore
\[
1 - e
 = 1 - k (f)
 = l (1 - f)
 \sim \big( 1 - f - \tfrac{2}{3} \big)_{+}
 \leq (1 - f - \ep)_{+}
 \precsim x
 \sim a,
\]
as desired.

Now suppose that $A$ is finite
and $\af$ has the generalized tracial Rokhlin property.
Let $\varepsilon$, $F$, $x$, and~$y$
be as in Definition~\ref{defwtrp}.
By \cite[Lemma 2.9]{Ph14},
there is $z \in \big( {\overline{x A x}} \big)_{+} \setminus \{ 0 \}$
such that whenever $c \in A_{+}$ satisfies $0 \leq c \leq 1$
and $c \precsim z$,
then $\| (1 - c) x (1 - c) \| > 1 - \ep$.
Apply \cite[Definition 5.2]{HO13},
getting
orthogonal positive contractions
$f_g \in A$ for $g \in G$ such that,
with $f = \sum_{g \in G} f_{g}$,
the following hold:
\begin{enumerate}
\item\label{P_8615_WeakToGen_Sub}
$1 - f \precsim z$.
\item\label{P_8615_WeakToGen_Comm}
$\| f_{g} a - a f_{g} \| < \varepsilon$
for all $a \in F$ and all $g \in G$.
\item\label{P_8615_WeakToGen_LeftTr}
$\|\alpha_{g} (f_{h}) - f_{g h} \| < \varepsilon$ for all $g, h \in G$.
\end{enumerate}
The last two conditions are
\eqref{WTRP-COM} and~\eqref{WTRP-LTR}
in Definition~\ref{defwtrp}.
For Condition~\eqref{WTRP-CU-EQ} in Definition~\ref{defwtrp},
we use
\eqref{P_8615_WeakToGen_Sub} above at the fourth step
and $z \in \big( {\overline{x A x}} \big)_{+}$
at the fifth step,
to get
\begin{align*}
\big( y^{2} - y f y - \varepsilon \big)_{+}
& \leq y (1 - f) y
\\
& \sim (1 - f)^{1/2} y^2 (1 - f)^{1/2}
  \leq \| y \|^2 (1 - f)
  \precsim z
  \precsim x.
\end{align*}
Condition~\eqref{WTRP-OUT-CO} in Definition~\ref{defwtrp}
follows from \eqref{P_8615_WeakToGen_Sub} above
and the choice of~$z$.
\end{proof}

\begin{lemma}\label{lem_fg}
Let $A$ be a simple tracially $\mathcal{Z}$-absorbing C*-algebra and
let $\alpha \colon G \to \mathrm{Aut} (A)$ be an action of
a finite group $G$ on $A$ with the weak tracial Rokhlin property.
Then for every finite set $F \subseteq A$,
every $\varepsilon > 0$, every  nonzero positive element
$a \in A$, every positive contraction $x \in A^{\alpha}$, and
every $n \in \mathbb{N}$, there is a c.p.c.~order zero
map $\psi \colon M_{n} \to A$ such that:
\begin{enumerate}[label=$\mathrm{(\arabic*)}$]
\item\label{TZS_WTRP_1}
$\big( x^{2} - x \psi (1) x - \varepsilon \big)_{+} \precsim a$.
\item\label{TZS_WTRP_2}
$\| [\psi (z), y] \| < \varepsilon$
for any $z \in M_{n}$ with $\| z \| \leq 1$ and any $y \in F$.
\item\label{TZS_WTRP_3}
$\| \alpha_{g} (\psi (z)) - \psi (z) \| < \varepsilon$
for any $z \in {M}_{n}$
with $\| z \| \leq 1$ and any $g \in G$.
\setcounter{TmpEnumi}{\value{enumi}}
\end{enumerate}
\end{lemma}

\begin{proof}
Let $F$, $\varepsilon$, $a$, $x$, and~$n$ be given as in the statement.
We may assume that $F \subseteq A_{+}$
and that $\| y \| \leq 1$ for all $y \in F$.
Choose $\delta > 0$ such that $\delta < \varepsilon / [7  \card (G)]$.

Use \cite[Proposition 2.5]{KW04}
to choose $\eta_0 > 0$ such that
whenever $\varphi \colon M_{n} \to A$
is a c.p.c.\  map
such that $\|\varphi (y) \varphi (z) \| < \eta_0$ for
all $y, z \in (M_{n})_{+}$ with $y z = 0$ and $\|y\|, \|z\| \leq 1$
(a c.p.c.\  $\et_0$-order zero map),
then there is a c.p.c.~order zero map
$\psi \colon M_{n} \to A$
such that $\|\varphi (z) - \psi (z) \| < \delta$
for all $z \in M_{n}$ with $\| z \| \leq 1$.

Use \cite[Lemma~2.5]{AP16} to choose
$\eta_1 > 0$ such that whenever $D$ is a C*-algebra
and $h, k \in D$ satisfy $0 \leq h, k \leq 1$
and $\| [h, k]\| < \eta_1$, then $\| [h^{1/2}, k] \| < \delta$.
Now set $\eta = \min (\eta_0, \eta_1, \delta)$.

We claim that there is $b \in A_{+} \setminus \{ 0 \}$ such that:
\begin{enumerate}
\setcounter{enumi}{\value{TmpEnumi}}
\item\label{TZS_WTRP_4}
$b \oplus \bigoplus_{g \in G} \alpha_{g} (b) \precsim a$.
\setcounter{TmpEnumi}{\value{enumi}}
\end{enumerate}
To prove the claim,
first set $k = \card (G) + 1$
and use \cite[Lemma~2.4]{Ph14} to find
$c \in A_{+} \setminus \{ 0 \}$ such that
$c \otimes 1_{k} \precsim a$ in $M_{\infty} (A)$.
Then \cite[Lemma~2.6]{Ph14} implies that there is
$b \in A_{+} \setminus \{ 0 \}$
such that $b \precsim \alpha_{g}^{-1} (c)$ for all $g \in G$.
Thus,
\[
b \oplus \bigoplus_{g \in G} \alpha_{g} (b)
 \precsim c \oplus \bigoplus_{g \in G} c
 = c \otimes 1_{k}
 \precsim a,
\]
which is~(\ref{TZS_WTRP_4}).

Since $\alpha$ has the weak tracial Rokhlin property,
there are orthogonal positive contractions $e_g \in A$ for $g \in G$
such that,
with $e = \sum_{g \in G} e_{g}$, the following hold:
\begin{enumerate}
\setcounter{enumi}{\value{TmpEnumi}}
\item\label{TZS_WTRP_5}
$\| e_{g} y - y e_{g} \| < \eta$
for all $y \in F \cup \{x\}$ and all $g \in G$.
\item\label{TZS_WTRP_6}
$\|\alpha_{g} (e_{h}) - e_{gh} \| < \eta$ for all $g, h \in G$.
\item\label{TZS_WTRP_7}
$\big( x^{2} - x e x - \eta \big)_{+} \precsim b$.
\setcounter{TmpEnumi}{\value{enumi}}
\end{enumerate}
Set
\[
E = \big\{\alpha_{g} (y) \colon g \in G,  \, y \in F \big\}
   \cup \big\{\alpha_{h} (e_{g}), \, \alpha_{h} (e_{g})^{1/2}
         \colon g, \, h \in G \big\},
\]
which is a finite subset of~$A$.
Since $A$ is tracially $\mathcal{Z}$-absorbing,
there is a c.p.c.~order zero map $\varphi \colon M_{n} \to A$
such that the following hold:
%
\begin{enumerate}
\setcounter{enumi}{\value{TmpEnumi}}
\item\label{TZS_WTRP_9}
$\big( x^{2} -x \varphi (1) x - \eta \big)_{+} \precsim b$.
\item\label{TZS_WTRP_8}
$\| [\varphi (z), y] \| < \eta$ for any
$z \in M_{n}$ with $\| z \| \leq 1$ and any $y \in E$.
\setcounter{TmpEnumi}{\value{enumi}}
\end{enumerate}

Define a c.p.c.~map $\widetilde{\varphi} \colon M_{n} \to A$ by
\[
\widetilde{\varphi} (z)
 = \sum_{g \in G} e_{g}^{1/2} \alpha_{g} (\varphi (z)) e_{g}^{1/2}
\]
for $z \in M_n$.
It follows from \eqref{TZS_WTRP_8}
that:
\begin{enumerate}
\setcounter{enumi}{\value{TmpEnumi}}
\item\label{TZS_WTRP_New}
$\big\| \widetilde{\varphi} (z)
    - \sum_{g \in G} e_{g} \alpha_{g} (\varphi (z)) \big\|
 < \eta \card (G)$
for any $z \in M_{n}$ with $\| z \| \leq 1$.
\setcounter{TmpEnumi}{\value{enumi}}
\end{enumerate}
%
For $y, z \in (M_{n})_{+}$
with $y z = 0$ and $\|y\|, \|z\| \leq 1$,
we use orthogonality of the elements $e_g$
at the first two steps,
$\varphi (y) \varphi (z) = 0$ at the third step,
and \eqref{TZS_WTRP_8} at the fourth step,
getting
\begin{align*}
\| \widetilde{\varphi} (y) \widetilde{\varphi} (z) \|
& = \bigg\| \sum_{g \in G}
   e_{g}^{1/2} \alpha_{g} (\varphi (y)) e_{g}^{1/2}
    \cdot e_{g}^{1/2} \alpha_{g} (\varphi (z)) e_{g}^{1/2} \bigg\|
\\
& = \max_{g \in G} \big\| e_{g}^{1/2} \alpha_{g} (\varphi (y))
    e_{g} \alpha_{g} (\varphi (z)) e_{g}^{1/2} \big\|
\\
& \leq \max_{g \in G} \big\|
 \big[ \alpha_{g}^{-1} (e_{g}), \, \varphi (y) \big] \big\|
  < \eta
  \leq \eta_0.
\end{align*}
By the choice of $\eta_0$, there is a c.p.c.~order zero map
$\psi \colon M_{n} \to A$ such that:
\begin{enumerate}
\setcounter{enumi}{\value{TmpEnumi}}
\item\label{TZS_WTRP_X}
$\|\widetilde{ \varphi} (z) - \psi (z) \| < \delta$
for any $z \in M_{n}$ with $\| z \| \leq 1$.
\setcounter{TmpEnumi}{\value{enumi}}
\end{enumerate}
We will show that $\psi$ has the desired properties
\ref{TZS_WTRP_1}, \ref{TZS_WTRP_2}, and~\ref{TZS_WTRP_3}
in the statement.

We prove~\ref{TZS_WTRP_1}.
Using~\eqref{TZS_WTRP_7} and \eqref{TZS_WTRP_9}
at the first step
and \eqref{TZS_WTRP_4} at the third step, we get
\begin{align}\label{lem_fg_eq1}
& \big( x^{2} - x e x - \eta \big)_{+}
  + \sum_{g \in G} e_{g}^{1/2}
      \alpha_{g}
          \big( \big( x^{2} - x \varphi (1) x - \eta \big)_{+} \big)
         e_{g}^{1/2}
\\
& \hspace*{3em} {\mbox{}}
   \precsim b \oplus \bigoplus_{g \in G}
              e_{g}^{1/2} \alpha_g (b) e_{g}^{1/2}
   \precsim b \oplus \bigoplus_{g \in G} \alpha_{g} (b)
   \precsim a.
\notag
\end{align}
On the other hand,
using the assumption that $x \in A^{\alpha}$ at the second step,
using \eqref{TZS_WTRP_5}, $\eta \leq \eta_1$,
and the choice of $\eta_1$ at the third step,
and using \eqref{TZS_WTRP_X} at the last step,
we get
\begin{align*}
& \big( x^{2} - x e x - \eta \big)_{+}
    + \sum_{g \in G} e_{g}^{1/2} \alpha_{g}
         \big( \big( x^{2} - x \varphi (1) x - \eta \big)_{+} \big)
      e_{g}^{1/2}
\\
& \hspace*{3em} {\mbox{}}
 \approx_{2 \eta} x^{2} - x e x + \sum_{g \in G} e_{g}^{1/2}
\alpha_{g} \big( x^{2} - x \varphi (1) x \big) e_{g}^{1/2}
\\
& \hspace*{3em} {\mbox{}}
 = x^{2} - x e x + \sum_{g \in G} e_{g}^{1/2} x^{2} e_{g}^{1/2}
    - \sum_{g \in G}
       e_{g}^{1/2} x \alpha_{g} (\varphi (1)) x e_{g}^{1/2}
\\
& \hspace*{3em} {\mbox{}}
 \approx_{4 \delta \card (G)} x^{2} - x e x + \sum_{g \in G} x e_{g} x
  - \sum_{g \in G} x e_{g}^{1/2} \alpha_{g} (\varphi (1)) e_{g}^{1/2} x
\\
& \hspace*{3em} {\mbox{}}
 = x^{2} - x \widetilde{\varphi} (1) x
\\
& \hspace*{3em} {\mbox{}}
 \approx_{\delta} x^{2} - x {\psi} (1) x.
\end{align*}
So
\begin{align*}
&
\Bigg\| x^{2} - x \psi (1) x
 - \Bigg[ (x^{2} - x e x - \eta)_{+} +
     \sum_{g \in G} e_{g}^{1/2}
        \alpha_{g} \big( (x^{2} - x \varphi (1) x - \eta \big)_{+} \big)
             e_{g}^{1/2} \Bigg] \Bigg\|
\\
& \hspace*{3em} {\mbox{}}
  < 2 \eta + 4 \delta  \card (G) + \delta
  \leq 7 \delta  \card (G)
  < \varepsilon.
\end{align*}
Therefore, using \eqref{lem_fg_eq1} at the second step, we have
\begin{align*}
& \big( x^{2} - x \psi (1) x - \varepsilon \big)_{+}
\\
& \hspace*{3em} {\mbox{}}
\precsim \big( x^{2} - x e x - \eta \big)_{+}
  + \sum_{g \in G} e_{g}^{1/2}
     \alpha_{g} \big( \big( x^{2}
        - x \varphi (1) x - \eta \big)_{+} \big) e_{g}^{1/2}
 \precsim a,
\end{align*}
as desired.

To prove \ref{TZS_WTRP_2}, let $y \in F$
and let $z \in M_{n}$ satisfy $\| z \| \leq 1$.
Using \eqref{TZS_WTRP_X} at the first step,
\eqref{TZS_WTRP_New} at the second step,
and \eqref{TZS_WTRP_5} and \eqref{TZS_WTRP_8} at the third step,
we have
\begin{align*}
\|[\psi (z), y] \|
& < 2 \delta + \|[\widetilde{\varphi} (z), y] \|
\\
& \leq  2 \delta + 2 \eta \card (G)
  + \Bigg\| \Bigg[
     \sum_{g \in G} e_{g} \alpha_{g} (\varphi (z)), y \Bigg] \Bigg\|
\\
& < 2 \delta + 2 \eta  \card (G)+ 2 \eta  \card (G)
  \leq 6 \delta \card (G)
  < \varepsilon.
\end{align*}

It remains to show \ref{TZS_WTRP_3}.
Let $g \in G$ and let $z \in M_{n}$ satisfy $\| z \| \leq 1$.
Using~\eqref{TZS_WTRP_X} at the first step,
\eqref{TZS_WTRP_New} at the second step,
and~\eqref{TZS_WTRP_6} at the third step we have:
\begin{align*}
\| \alpha_{g} (\psi (z))
& - \psi (z) \|
\\
& < 2 \delta + \| \alpha_{g} (\widetilde{\varphi} (z))
  - \widetilde{\varphi} (z) \|
\\
& \leq 2 \delta +2 \eta   \card (G) +
\Bigg\| \sum_{h \in G} \alpha_{g} (e_{h}) \alpha_{gh} (\varphi (z)) -
\sum_{h \in G} e_{gh} \alpha_{gh} (\varphi (z)) \Bigg\|
\\
& \leq 2 \delta + 2 \eta \card (G) + \eta \card (G)
  \leq 5 \delta  \card (G)
  < \varepsilon.
\end{align*}
This finishes the proof.
\end{proof}

By averaging over the group
and applying \cite[Proposition 2.5]{KW04} again,
we can improve the statement of Lemma~\ref{lem_fg}
to require that $\alpha_g (\psi (z)) = \psi (z)$
for all $z \in M_n$ and all $g \in G$.
It seems simpler to do without this step.
Moreover, the way we do it,
the proof of Theorem~\ref{thm_fg} is a better model for
the proof of Theorem~\ref{thm_int}, the case $G = \mathbb{Z}$.

\begin{theorem}\label{thm_fg}
Let $\alpha \colon G \to \mathrm{Aut} (A)$ be an
action of a finite group $G$ on a simple (not necessarily unital)
tracially $\mathcal{Z}$-absorbing C*-algebra $A$.
If $\alpha$ has the weak tracial Rokhlin
property, then $C^{*} (G, A, \alpha)$ and $A^{\alpha}$ are also  simple
tracially $\mathcal{Z}$-absorbing C*-algebras.
\end{theorem}

\begin{proof}
By \cite[Proposition~3.2]{FG17},
$\alpha$ is pointwise outer.
So \cite[Theorem 3.1]{Ki81}
implies that $C^{*} (G, A, \alpha)$ is simple.
Since $A^{\alpha}$ is isomorphic
to a hereditary C*-subalgebra of $C^{*} (G, A, \alpha)$ \cite{Ro79},
it follows that $A^{\alpha}$ is also simple,
and by \cite[Theorem~4.1]{AGJP21}, it also follows that it is
enough to show that $C^{*} (G, A, \alpha)$ is
tracially $\mathcal{Z}$-absorbing.

To verify \cite[Definition~3.6]{AGJP21} for $C^{*} (G, A, \alpha)$,
let $F \subseteq C^{*} (G, A, \alpha)$ be a finite set,
let $x, a \in C^{*} (G, A, \alpha)_{+}$ with $a \neq0$,
let $\varepsilon > 0$, and let $n \in \mathbb{N}$.
The proof of \cite[Lemma~5.1]{HO13}
also works in the nonunital case,
and we can apply this generalization of it
to find $a_{0} \in A_{+} \setminus \{ 0 \}$
such that $a_{0} \precsim a$ in $C^{*} (G, A, \alpha)$.
Next,
let $(u_{i})_{i \in I}$ be an approximate identity for~$A$.
For $i \in I$ set
\[
v_{i}
 = \frac{1}{\card (G)} \sum_{g \in G} \alpha_{g} (u_{i})
 \in A^{\alpha}.
\]
Then $(v_{i})_{i \in I}$
is an approximate identity for $C^{*} (G, A, \alpha)$
which is contained in $A^{\alpha}$.
Therefore, by \cite[Remark~3.8]{AGJP21},
we may assume that $x \in A^{\alpha}$.

For $g \in G$ let $u_g \in M ( C^{*} (G, A, \alpha) )$
be the standard unitary
associated with the crossed product.
For $y \in F$ write
$y = \sum_{g \in G} c_{y, g} u_g$ with $c_{y, g} \in A$ for
$g \in G$.
Define
$E = \big\{ c_{y, g} \colon {\mbox{$y \in F$ and $g \in G$}} \big\}$.
Apply Lemma~\ref{lem_fg} with $E$ in place of $F$,
with $a_{0}$ in place of $a$, with $x$ as given,
and with $\varepsilon / [2 \card (G) ]$ in place of $\varepsilon$.
We obtain a c.p.c.~order zero map $\psi_0 \colon M_{n} \to A$.
Let $\psi \colon A \to C^{*} (G, A, \alpha)$ be its composition
with the inclusion of $A$ in $C^{*} (G, A, \alpha)$.
We claim that $\psi$ satisfies
the conditions of \cite[Definition~3.6]{AGJP21},
for $C^{*} (G, A, \alpha)$.
Condition~(1) is clear.
For Condition~(2),
let $z \in M_n$ satisfy $\| z \| \leq 1$ and let $y \in F$.
Then
\begin{align*}
\| [ \psi (z), \, y ] \|
& \leq \sum_{g \in G} \| [ \psi (z), \, c_{y, g} u_g ] \|
\\
& \leq \sum_{g \in G} \| [ \psi (z), \, c_{y, g} ] \|
      + \sum_{g \in G} \| [ \psi (z), \, u_g ] \|
\\
&  < \card (G) \left( \frac{\ep}{2 \card (G)} \right)
       + \card (G) \left( \frac{\ep}{2 \card (G)} \right)
   = \ep.
\end{align*}
This completes the proof.
\end{proof}

\begin{remark}\label{rmk_fg}
The proofs of Lemma~\ref{lem_fg} and Theorem~\ref{thm_fg}
work
if moreover $A$ is unital and
$\alpha$ has the generalized tracial Rokhlin property
(\cite[Definition~5.2]{HO13})
instead of having the weak tracial Rokhlin property.
This is because Condition~\eqref{WTRP-OUT-CO}
in Definition~\ref{defwtrp} is not used in the proof of
Lemma~\ref{lem_fg}.

The point is that the definitions are written
so that any purely infinite simple C*-algebra
is tracially $\mathcal{Z}$-absorbing,
but not every action of a finite group
on a purely infinite simple C*-algebra
has the weak tracial Rokhlin property.
The generalized tracial Rokhlin property
of~\cite{HO13}
is strong enough to imply that the crossed product is simple.
By \cite[Theorem 4.5]{JngOsk},
the crossed product of a purely infinite simple unital C*-algebra
by a pointwise outer action of a finite group
is always purely infinite and simple.
In fact,
even if the action isn't pointwise outer,
the crossed product is a finite direct sum
of purely infinite simple C*-algebras,
so is purely infinite in the sense of
\cite[Definition 4.1]{KR00}.

We thus obtain an alternative proof for
Theorem~5.6 of \cite{HO13} which avoids dimension functions.
In particular,  the assumption of separability in
Theorem~5.6 of \cite{HO13}
is unnecessary.
(This assumption
is implicit in the statement of that result
because the authors use their Lemma~5.5.)
\end{remark}

\begin{corollary}\label{Cor_intermediate}
Let  $\alpha \colon G \to \mathrm{Aut} (A)$ be an
action of a finite group $G$ on a 
simple $\sigma$-unital  tracially $\mathcal{Z}$-absorbing 
 C*-algebra $A$.
If $\alpha$ has the weak tracial Rokhlin property 
then all intermediate C*-algebras $B$ and $D$ with 
\[
A^{\alpha}  \subseteq B \subseteq A \subseteq D \subseteq C^*(G, A,\alpha)
\]
are   simple and
tracially $\mathcal{Z}$-absorbing.
\end{corollary}

\begin{proof}
Suppose that $B$ and $D$ are intermediate C*-algebras as above.
By \cite[Proposition~3.2]{FG17}, $\alpha$ is pointwise outer.
Then, \cite[Corollary~6.6(1)]{Izumi2002} implies that there is a subgroup $H$ of
$G$ such that  $D=  C^{*} (H, A, \beta)$
where $\beta \colon H \to \mathrm{Aut} (A)$ is the 
restriction of $\alpha$ to $H$.
By \cite[Proposition~4.1]{FG17}, $\beta$ has 
the weak tracial Rokhlin property. Now, 
Theorem~\ref{thm_fg} yields that $D$ is simple
and tracially $\mathcal{Z}$-absorbing.
The corresponding result about $B$ follows similarly using
the fact that $B$ is equal  to the fixed point
algebra of a restricted action $\beta \colon H \to \mathrm{Aut} (A)$
as above \cite[Corollary~6.6(2)]{Izumi2002}.
\end{proof}

\begin{corollary}\label{cor_fg_zs}
Let $\alpha \colon G \to \mathrm{Aut} (A)$ be an
action of a finite group $G$ on a
simple separable nuclear $\mathcal{Z}$-stable C*-algebra $A$.
If $\alpha$ has the weak tracial Rokhlin property,
then $C^{*} (G, A, \alpha)$ and $A^{\alpha}$ are
simple $\mathcal{Z}$-stable C*-algebras.
Moreover, all intermediate C*-algebras $B$ and $D$
as in Corollary~\ref{Cor_intermediate} are simple, nuclear, and
$\mathcal{Z}$-stable.
\end{corollary}

\begin{proof}
By Theorem~\ref{thm_fg},
  $C^{*} (G, A, \alpha)$ is  
simple and tracially $\mathcal{Z}$-absorbing. 
It is also nuclear since $A$ is nuclear and $G$ is amenable.
By the main result of \cite{CLS2021}, 
we see that $C^{*} (G, A, \alpha)$ is $\mathcal{Z}$-stable.
On the other hand,
$A^{\alpha}$ is isomorphic
to a full corner of $C^{*} (G, A, \alpha)$ \cite{Ro79},
and
by \cite[Corollary~3.2]{TW07},
$\mathcal{Z}$-absorption is preserved under Morita equivalence
in the class of separable C*-algebras.
Therefore $A^{\alpha}$ is also a
simple $\mathcal{Z}$-stable C*-algebra.

The second part about intermediate C*-algebras follows
from the first part and an argument similar to the proof
of Corollary~\ref{Cor_intermediate}.
\end{proof}

The following example can be thought of as
the analog of \cite[Example 5.10]{HO13}.
It is also a consequence of Theorem~\ref{T_8816_PermwTRP} below.

\begin{example}\label{E_8703_WShift}
Let $\mathcal{W}$ be the Razak-Jacelon algebra~\cite{Jcln},
that is, the unique simple Razak algebra~\cite{Rzk}
which has a unique tracial state and no unbounded traces.
Let $n \in \{ 2, 3, \ldots \}$,
and let $A = \mathcal{W}^{\otimes n}$,
the tensor product of $n$ copies of~$\mathcal{W}$.
We have $A \cong \mathcal{W}$ by \cite[Corollary 19.3]{GngLin}
and $A \cong A \otimes {\mathcal{Z}}$ by \cite[Corollary 6.3]{Jcln}.

Let $\af$ be the action of the symmetric group~$S_n$
on $A$ by permuting the tensor factors.
We claim that $\af$ has the weak tracial Rokhlin property.
First, let $\bt$ be the action of~$S_n$
on ${\mathcal{Z}}^{\otimes n}$ by permuting the tensor factors.
By \cite[Example 5.10]{HO13},
this action has the generalized Rokhlin paperty
of \cite[Definition 5.2]{HO13}.
By Proposition \ref{P_8615_WeakToGen}\ref{P_8615_WeakToGen_ToWk},
it has the weak tracial Rokhlin property.
Therefore the action $\af \otimes \bt$
of $S_n$ on
$A \otimes {\mathcal{Z}}^{\otimes n}
 = (\mathcal{W} \otimes {\mathcal{Z}})^{\otimes n}$
has the weak tracial Rokhlin property
by \cite[Proposition 4.5]{FG17}.
Using an isomorphism $\mathcal{W} \cong \mathcal{W} \otimes {\mathcal{Z}}$
(\cite[Corollary 6.3]{Jcln}),
one sees that $\af$ is conjugate to $\af \otimes \bt$,
and so has the weak tracial Rokhlin property.
Alternatively, it follows from Theorem~\ref{T_8816_PermwTRP} below 
that $\af$ has the weak tracial Rokhlin property.

It follows from Theorem~\ref{thm_fg}
that $C^* (S_n, \mathcal{W}^{\otimes n}, \af)$
is tracially $\mathcal{Z}$-absorbing.
\end{example}

We don't know whether the action in Example~\ref{E_8703_WShift}
has the Rokhlin property.
The corresponding action on~${\mathcal{Z}}$
does not have even any higher dimensional Rokhlin property
with commuting towers,
by \cite[Corollary 4.8]{HrsPh1}.
However,
the proof relies,
among other things, on nontriviality of $K_0 ( {\mathcal{Z}} )$,
while $K_* (\mathcal{W}) = 0$.

\begin{example}\label{exa_fg}
Let $A = \bigotimes_{k = 1} ^{\infty} M_{3}$
be the UHF~algebra of type $3^{\infty}$ and
let $B$ be a (not necessarily unital) simple C*-algebra.
Then $A \otimes B$ is $\mathcal{Z}$-stable,
so is tracially $\mathcal{Z}$-absorbing.
Consider the action
$\alpha \colon \mathbb{Z}_{2} \to \mathrm{Aut} (A)$ generated by
the automorphism
\[
\bigotimes_{k = 1}^{\infty}
  \mathrm{Ad}
   \begin{pmatrix}1&0&0 \cr 0&1&0 \cr 0&0& - 1 \end{pmatrix}
 \in \mathrm{Aut} (A),
\]
and let $\beta \colon \mathbb{Z}_{2} \to \mathrm{Aut} (B)$
be an arbitrary action.
The action $\alpha$ has the tracial Rokhlin property
(for unital C*-algebras and using projections)
by the condition in \cite[Proposition 2.5(3)]{PhT4}
(but not the Rokhlin property,
by the condition in \cite[Proposition 2.4(3)]{PhT4}).
By \cite[Proposition~4.5]{FG17}, the action
$\alpha \otimes \beta
 \colon \mathbb{Z}_{2} \to \mathrm{Aut} (A \otimes B)$
has the weak tracial Rokhlin property.
Now Theorem~\ref{thm_fg} implies that
$C^*(\mathbb{Z}_{2}, \, A \otimes B, \, \alpha \otimes \beta)$
is a simple tracially $\mathcal{Z}$-absorbing
C*-algebra.

As an example for $B$ and $\bt$,
let $n \in \{ 2, 3, \ldots, \infty \}$,
take $B$ to be the minimal tensor product of two copies
of a proper hereditary subalgebra of $C^*_{\mathrm{r}} (F_n)$,
and take $\bt$ to be generated by the (minimal) tensor flip.
\end{example}

A more interesting specific case
would be gotten by taking $B$ to be
the reduced free product of two copies of
the same nonunital C*-algebra,
and taking $\bt$ to be the free flip.
Criteria for simplicity of reduced free products
of unital C*-algebras are given in
the corollary to \cite[Proposition 3.1]{Avtz},
but unfortunately we do not know criteria
for simplicity of reduced free products
of nonunital C*-algebras.

\section{The permutation action on a finite tensor product}\label{Sec_Perm}

\indent
We will prove that if $A$
is a simple tracially $\mathcal{Z}$-absorbing C*-algebra,
then, under a finiteness condition,
the permutation action of the symmetric group
on the minimal tensor product of finitely many copies of~$A$
has the weak tracial Rokhlin property
of Definition~\ref{defwtrp}.
The basic idea is simple.
Fix $m \in \N$,
the number of tensor factors.
Suppose $n \in \N$ and $\ph \colon M_{n} \to A$
is a c.p.c.~order zero map.
Let $(e_{j, k})_{j, k = 1, 2, \ldots, m}$
be the standard system of matrix units for~$M_m$.
Then in the tensor product of $m$ copies of~$A$,
we can look at the elements
\[
\ph (e_{r (1), r (1)}) \otimes \ph (e_{r (2), r (2)})
 \otimes \cdots \otimes \ph (e_{r (m), r (m)})
\]
with
\[
r = (r (1), \, r (2), \, \ldots, \, r (m)) \in \{ 1, 2, \ldots n \}^m.
\]
The permutation action on the tensor product
translates into permutation of the indices
in this expression.
If $n$ is large,
then most elements $r \in \{ 1, 2, \ldots n \}^m$
have all coordinates distinct.
The action is free on the set of such~$r$,
and adding up one of these from each orbit
gives the elements required in the weak tracial Rokhlin property.

Some condition is presumably needed.
The way the definitions are written,
every purely infinite simple C*-algebra
is tracially $\mathcal{Z}$-absorbing,
but not every action on such a C*-algebra
has the weak tracial Rokhlin property.

The details are somewhat messy,
so we start with some lemmas.

\begin{notation}\label{N_8808_Setup}
For $m \in \N$
let $S_m$ denote the corresponding symmetric group,
the group of all permutations of $\{ 1, 2, \ldots m \}$.
Let $(e_{j, k})_{j, k = 1, 2, \ldots, m}$
be the standard system of matrix units for~$M_m$.
For a C*-algebra~$A$,
we let $A^{\otimes m}$ denote the minimal
tensor product of $m$ copies of~$A$,
and for $x \in A$ we write
$x^{\otimes m}$
for the $m$-fold tensor product
$x \otimes x \otimes \cdots \otimes x \in A^{\otimes m}$.
The permutation action $\bt$ of $S_m$ on $A^{\otimes m}$
is determined by
\[
\bt_{\sm} (a_1 \otimes a_2 \otimes \cdots \otimes a_m)
 = a_{\sm^{-1} (1)} \otimes a_{\sm^{-1} (2)}
    \otimes \cdots \otimes a_{\sm^{-1} (m)}
\]
for $a_1, a_2, \ldots, a_m \in A$
and $\sm \in S_m$.
(We suppress $A$ and $m$ in the notation.)
Finally, we recall \cite[Notation~7.3]{AGJP21}:
the Pedersen ideal of $A$ is $\Ped (A)$.
\end{notation}

\begin{lemma}\label{L_8825_yStarby}
Let $A$ be a C*-algebra,
let $b \in A_{+}$,
let $a \in \bigl( {\overline{A b A}} \bigr)_{+}$,
and let $\ep > 0$.
Then there are $n \in \N$ and
$x_1, x_2, \ldots, x_n \in A$
such that $(a - \ep)_{+} = \sum_{j = 1}^n x_j^* b x_j$.
\end{lemma}

\begin{proof}
Use \cite[Lemma 1.13]{Ph14} to find
$y_1, y_2, \ldots, y_n \in A$
such that
\[
\Bigg\| a - \sum_{j = 1}^n y_j^* b y_j \Bigg\| < \ep.
\]
\cite[Lemma~2.1]{AGJP21}, provides $d \in A$ such that
\[
(a - \ep)_{+} = d^* \Bigg( \sum_{j = 1}^n y_j^* b y_j \Bigg) d.
\]
Set $x_j = y_j d$ for $j = 1, 2, \ldots, n$.
\end{proof}

\begin{lemma}\label{L_8825_DFSub}
Let $A$ be a simple tracially $\mathcal{Z}$-absorbing C*-algebra,
let $a \in A_{+}$,
let $b \in A_{+} \setminus \{ 0 \}$,
and let $\ep > 0$.
Then there is $n_0 \in \N$
such that whenever $\nu \in \N$
and $c \in A_{+}$
satisfy $n_0 \nu \langle c \rangle \leq \langle (a - \ep)_{+} \rangle$,
then
$\nu \langle c \rangle \leq \langle b \rangle$.
\end{lemma}

One can prove this lemma using dimension functions:
since $(a - \ep)_{+} \in \Ped (A)$,
the infimum of $d ((a - \ep)_{+})$ over suitably
normalized dimension functions will be strictly positive,
and we can choose $\nu$ to be greater than the reciprocal
of this infimum.
However,
the result can be gotten directly from
almost unperforation.

\begin{proof}[Proof of Lemma~\ref{L_8825_DFSub}]
Lemma~\ref{L_8825_yStarby}
provides $n \in \N$ and
$x_1, x_2, \ldots, x_n \in A$
such that $(a - \ep)_{+} = \sum_{j = 1}^n x_j^* b x_j$.
Then, in $M_n (A)$,
\begin{align*}
(a - \ep)_{+}
& \precsim
  \diag ( x_1^* b x_1, \, x_2^* b x_2, \, \ldots, \, x_n^* y x_n )
\\
& \sim \diag \bigl( b^{1/2} x_1 x_1^* b^{1/2},
     \, b^{1/2} x_2 x_2^* b^{1/2}, \, \ldots,
     \, b^{1/2} x_n x_n^* b^{1/2} \bigr)
  \precsim 1_n \otimes b.
\end{align*}

Set $n_0 = n + 1$.
If $n_0 \nu \langle c \rangle \leq \langle (a - \ep)_{+} \rangle$,
then $(n + 1) \nu \langle c \rangle \leq n \langle b \rangle$,
so $\nu \langle c \rangle \leq \langle b \rangle$
by \cite[Theorem~6.4]{AGJP21}.
\end{proof}

\begin{lemma}\label{L_8808_CommEqEst}
For every $\ep > 0$ there is $\dt > 0$ such that the following holds.
Let $n \in \N$,
let $A$ be a C*-algebra,
let $x \in A$ satisfy $\| x \| \leq 1$,
and let $\varphi \colon M_{n} \to A$ be a c.p.c.~order zero map
such that
$\| [\varphi (z), x] \| < \dt$
for any $z \in M_{n}$ with $\| z \| \leq 1$.
Then for every $j, k \in \{ 1, 2, \ldots, n \}$
there is $v \in A$ such that $\| v \| \leq 1$
and
$\bigl\| v^* x^* \ph (e_{j, j}) x v - x^* \ph (e_{k, k}) x \bigr\|
 < \ep$.
\end{lemma}

\begin{proof}
Define $f \colon [0, 1] \to [0, 1]$ by
\[
f (\ld)
 = \begin{cases}
   0                            &
                   \hspace*{1em} 0 \leq \ld \leq \frac{\ep}{8}
        \\
   8 \ep^{-1} \ld - 1           &
                   \hspace*{1em} \frac{\ep}{8} < \ld < \frac{\ep}{4}
       \\
   1                            &
                   \hspace*{1em} \frac{\ep}{4} \leq \ld \leq 1.
\end{cases}
\]
Then $f \in C_0 ( (0, 1] )$.
Further let $t \in C_0 ( (0, 1] )$ be the function
$t (\ld) = \ld$ for $\ld \in (0, 1]$.
By \cite[Lemma 2.5]{ABP16},
the unitized cone $(C M_2)^{+}$
is generated by $1$ and the elements
$t \otimes e_{j, k}$ for $j, k = 1, 2$.
(There is a misprint in \cite[Lemma 2.5]{ABP16}:
the word ``unital'' is missing.)
Therefore there is $\dt > 0$ such that whenever $A$ is a C*-algebra,
$x \in A$ satisfies $\| x \| \leq 1$,
$\ps \colon (C M_2)^{+} \to A^{+}$
is a unital homomorphism,
and $\| [x, \, \ps ( t \otimes e_{j, k} ) ] \| < \dt$
for $j, k = 1, 2$,
then
$\| [x, \, \ps ( f \otimes e_{1, 2} ) ] \| < \frac{\ep}{4}$.

Now let $n$, $A$, $x$, $\ph$, $j$, and~$k$ be as in the hypotheses.
There is nothing to prove if $j = k$,
so assume $j \neq k$.
Let $\ps \colon C M_n \to A$ be the corresponding homomorphism
(\cite[Corollary~4.1]{WZ09}),
satisfying $\ps (t \otimes z) = \ph (z)$ for $z \in M_n$,
and let $\ps^{+} \colon (C M_n)^{+} \to A^{+}$ be its unitization.
Then
$\| [x, \, \ps^{+} ( t \otimes e_{j, k} ) ] \| < \dt$
for $j, k = 1, 2, \ldots, n$.
By considering the cone over the embedding of $M_2$ in $M_n$
determined by
\[
e_{1, 1} \mapsto e_{j, j},
\qquad
e_{1, 2} \mapsto e_{j, k},
\qquad
e_{2, 1} \mapsto e_{k, j},
\andeqn
e_{2, 2} \mapsto e_{k, k},
\]
we see that $v = \ps^{+} ( f \otimes e_{j, k} )$
satisfies
$\| [x, v] \| < \frac{\ep}{4}$.
Also clearly $\| v \| \leq 1$.

One checks that
\[
(f \otimes e_{j, k})^* (t \otimes e_{j, j}) (f \otimes e_{j, k})
 = t f^2 \otimes e_{k, k}
\]
and
$| \ld f (\ld)^2 - \ld | \leq \frac{\ep}{2}$
for all $\ld \in [0, 1]$,
so
$\| v^* \ph (e_{j, j}) v - \ph (e_{k, k}) \| \leq \frac{\ep}{2}$.
Now
\begin{align*}
\bigl\| v^* x^* \ph (e_{j, j}) x v - x^* \ph (e_{k, k}) x \bigr\|
& \leq \| [ v^*, x^* ] \| \cdot \| \ph (e_{j, j}) x v \|
        + \| x^*v^* \ph (e_{j, j}) \|\cdot \| [v, x] \|
\\
& \hspace*{3em} {\mbox{}}
        + \| x^* \|\cdot \| v^* \ph (e_{j, j}) v - \ph (e_{k, k}) \|\cdot \| x \|
\\
& < \frac{\ep}{4} + \frac{\ep}{4} + \frac{\ep}{2}
  = \ep,
\end{align*}
as desired.
\end{proof}

\begin{lemma}\label{L_8808_EltsSmall}
Let $A$ be a simple tracially $\mathcal{Z}$-absorbing C*-algebra.
Then
for every $x, a \in A_{+}$ with $a \neq 0$
and every $\varepsilon > 0$,
there is $n_0 \in \N$
such that for all $\nu \in \mathbb{N}$, all $\rh > 0$,
and every finite set $F \subseteq A$,
there is a c.p.c.~order zero map $\varphi \colon M_{n_0 \nu} \to A$
such that:
\begin{enumerate}
\item\label{deftza-it1_Small}
$\bigl( x^{2} - x \varphi (1) x - \varepsilon \bigr)_{+} \precsim a$.
\item\label{deftza-it2_Small}
$\| [\varphi (z), b] \| < \rh$
for any $z \in M_{n_0 \nu}$ with $\| z \| \leq 1$ and any $b \in F$.
\item\label{deftza-it3_Small}
$\nu \bigl\langle ( x \ph (e_{1, 1}) x - \varepsilon )_{+} \bigr\rangle
 \leq \langle a \rangle$
in $W (A)$.
\end{enumerate}
\end{lemma}

\begin{proof}
We may assume that $\|x\|\leq 1$, since if $\|x\|>1$ we can replace $x$ with
$x/\|x\|$ and $\varepsilon$ with $\varepsilon/\|x\|^2$ in \eqref{deftza-it1_Small}
and \eqref{deftza-it3_Small}, and then use
$(Mw-M\varepsilon)_+ \sim (w-\varepsilon)_+$ for all $w\in A_+$ and all $M>0$.
Set $x_0 = \bigl( x - \frac{\ep}{9} \bigr)_{+}$.
We require $x_0 \in F$.
Apply Lemma~\ref{L_8825_DFSub}
with $x$ in place of~$a$,
with $a$ in place of~$b$,
and with $\frac{\ep}{9}$ in place of~$\ep$,
getting $n_0 \in \N$.

Let $\nu \in \N$.
Apply Lemma~\ref{L_8808_CommEqEst}
with $\frac{\ep}{2}$ in place of~$\ep$,
and call the resulting number~$\rh_0$.
Use \cite[Lemma 2.5]{AP16} to choose $\rh_1 > 0$
such that whenever
$r, s \in A_{+}$
satisfy
\[
\| r \| \leq 1,
\qquad
\| s \| \leq 1,
\qquad {\mbox{and}} \qquad
\| r s - s r \| < \rh_1 ,
\]
then
\[
\bigl\| r^{1/2} s - s r^{1/2} \bigr\| < \frac{\ep}{4 n_0 \nu}.
\]
Set $\dt = \min (\rh, \rh_0, \rh_1)$.

Apply the definition of tracial $\mathcal{Z}$-absorption
(\cite[Definition~3.6]{AGJP21})
to find a c.p.c.~order zero map $\varphi \colon M_{n_0 \nu} \to A$
such that:
\begin{enumerate}
\item\label{8808_Use_deftza-it1}
$\bigl( x_0^{2} - x_0 \varphi (1) x_0
    - \tfrac{\varepsilon}{2} \bigr)_{+} \precsim a$.
\item\label{8808_Use_deftza-it2}
$\| [\varphi (z), b] \| < \dt$
for any $z \in M_{n_0 \nu}$ with $\| z \| \leq 1$
and any $b \in F  $.
\end{enumerate}

We have $\| x_0 - x \| < \frac{\ep}{8}$ and $\|x\|,\|x_0\|\leq 1$,
whence
\[
\bigl\| \bigl( x_0^2 - x_0 \ph (1) x_0 \bigr)
   - \bigl( x^2 - x \ph (1) x \bigr) \bigr\| < \frac{\ep}{2}.
\]
Then \cite[Corollary~1.6]{Ph14} implies
\[
\bigl( x^2 - x \ph (1) x - \ep \bigr)_{+}
 \precsim \bigl( x_0^2 - x_0 \ph (1) x_0 - \tfrac{\ep}{2} \bigr)_{+}
 \precsim a.
\]
This is part~\eqref{deftza-it1_Small}
of the conclusion.
Since $\dt \leq \rh$,
part~\eqref{deftza-it2_Small}
of the conclusion is clear.

It remains to prove part~\eqref{deftza-it3_Small}.
Since $x_0 \in F$,
the choice of $\rh_0$ using Lemma~\ref{L_8808_CommEqEst}
provides, for $k = 1, 2, \ldots, n_0 \nu$,
an element $v_k \in A$ such that
\[
\bigl\| v_k^* x_0 \ph (e_{k, k}) x_0 v_k
   - x_0 \ph (e_{1, 1}) x_0 \bigr\| < \frac{\ep}{2}
\andeqn
\| v_k \| \leq 1.
\]
Then, using \cite[Corollary~1.6]{Ph14} at the first step and
\cite[Lemma~2.3]{AGJP21} and $\| v_k \| \leq 1$ at the second step,
\begin{equation}\label{Eq_8808_HalfEst}
\bigl( x_0 \ph (e_{1, 1}) x_0 - \ep \bigr)_{+}
  \precsim \bigl( v_k^* x_0 \ph (e_{k, k}) x_0 v_k
      - \tfrac{\ep}{2} \bigr)_{+}
  \precsim \bigl( x_0 \ph (e_{k, k}) x_0
      - \tfrac{\ep}{2} \bigr)_{+}.
\end{equation}

Since $\dt \leq \rh_1$,
the choice of $\rh_1$ implies that
\[
\bigl\| \ph (e_{k, k})^{1/2} x_0 - x_0 \ph (e_{k, k})^{1/2} \bigr\|
 < \frac{\ep}{4 n_0 \nu}
\]
for $k = 1, 2, \ldots, n_0 \nu$.
So
\begin{align}\label{Eq_8817_NStar}
& \Biggl\| x_0 \ph (1) x_0
  - \sum_{k = 1}^{n_0 \nu}
       \ph (e_{k, k})^{1/2} x_0^2 \ph (e_{k, k})^{1/2} \Biggr\|
\\
& \hspace*{3em} {\mbox{}}
\leq \sum_{k = 1}^{n_0 \nu}
  \bigl\| x_0 \ph (e_{k, k}) x_0
       - \ph (e_{k, k})^{1/2} x_0^2 \ph (e_{k, k})^{1/2} \bigr\|
\notag
\\
& \hspace*{3em} {\mbox{}}
\leq \sum_{k = 1}^{n_0 \nu} \Bigl(
   \bigl\| \bigl[x_0, \ph (e_{k, k})^{1/2} \bigr] \bigr\|\cdot
            \bigl\| \ph (e_{k, k})^{1/2} x_0 \bigr\|
\notag
\\
& \hspace*{6em} {\mbox{}}
      + \bigl\| \ph (e_{k, k})^{1/2}x_0  \bigr\| \cdot
           \bigl\| \bigl[ \ph (e_{k, k})^{1/2}, x_0 \bigr] \bigr\|
    \Bigr)
\notag
\\
& \hspace*{3em} {\mbox{}}
 < 2 n_0 \nu \left( \frac{\ep}{4 n_0 \nu} \right)
 = \frac{\ep}{2}.
\notag
\end{align}
Since the elements
$\ph (e_{k, k})^{1/2} x_0^2 \ph (e_{k, k})^{1/2}$
are orthogonal,
we have
\[
\Biggl( \sum_{k = 1}^{n_0 \nu}
   \ph (e_{k, k})^{1/2} x_0^2 \ph (e_{k, k})^{1/2}
      - \frac{\ep}{2} \Biggr)_{+}
  = \sum_{k = 1}^{n_0 \nu}
      \bigl( \ph (e_{k, k})^{1/2} x_0^2 \ph (e_{k, k})^{1/2}
      - \tfrac{\ep}{2} \bigr)_{+}.
\]
Also by orthogonality,
and by~(\ref{Eq_8817_NStar}),
\[
\sum_{k = 1}^{n_0 \nu}
  \bigl\langle \bigl( \ph (e_{k, k})^{1/2} x_0^2 \ph (e_{k, k})^{1/2}
      - \tfrac{\ep}{2} \bigr)_{+} \bigr\rangle
 \leq \langle x_0 \ph (1) x_0 \rangle
 \leq \langle x_0 \rangle.
\]
By \cite[Proposition 2.3(ii)]{ERS},
we have
\[
\bigl( \ph (e_{k, k})^{1/2} x_0^2 \ph (e_{k, k})^{1/2}
      - \tfrac{\ep}{2} \bigr)_{+}
   \sim \bigl( x_0 \ph (e_{k, k}) x_0
      - \tfrac{\ep}{2} \bigr)_{+}.
\]
Therefore, using~\eqref{Eq_8808_HalfEst},
it follows that
\[
n_0 \nu \bigl\langle \bigl( x_0 \ph (e_{1, 1}) x_0 - \ep \bigr)_{+}
           \bigr\rangle
 \leq \langle x_0 \rangle.
\]
Now the choice of $n_0$ using Lemma~\ref{L_8825_DFSub}
implies
$\nu \bigl\langle ( x \ph (e_{1, 1}) x - \varepsilon )_{+} \bigr\rangle
 \leq \langle a \rangle$,
as desired.
\end{proof}

\begin{lemma}\label{L_8808_Switch}
Let $A$ and $B$ be simple C*-algebras,
with $B$ not of type~I.
Let $a, x \in A_{+} \setminus \{ 0 \}$,
and let $b \in B_{+} \setminus \{ 0 \}$.
Then for every $\ep > 0$
there is
$y \in \bigl( {\overline{b B b}} \bigr)_{+} \setminus \{ 0 \}$
such that
$(a - \ep)_{+} \otimes y \precsim x \otimes b$
in $A \otimes_{\min} B$.
\end{lemma}

\begin{proof}
By \cite[Lemma 1.13]{Ph14},
there are $n \in \N$ and $c_1, c_2, \ldots, c_n \in A$
such that $\left\| a - \sum_{j = 1}^n c_j^* x c_j \right\| < \ep$.
Use \cite[Lemma~2.1]{Ph14} to find orthogonal
Cuntz equivalent elements
$y_1, y_2, \ldots, y_n \in ({\overline{b B b}} )_{+} \setminus \{ 0 \}$.
Set $y = y_1$.
Then in $W (A \otimes_{\min} B)$ we have
\begin{align*}
\bigl\langle (a - \ep)_{+} \otimes y \bigr\rangle
& \leq \Biggl\langle \sum_{j = 1}^n c_j^* x c_j \otimes y \Biggr\rangle
  \leq \sum_{j = 1}^n \bigl\langle c_j^* x c_j \otimes y \bigr\rangle
\\
& = \sum_{j = 1}^n \langle x^{1/2} c_j c_j^* x^{1/2} \otimes y \rangle
  \leq n \langle x \otimes y \rangle
  = \sum_{j = 1}^n \langle x \otimes y_j \rangle
  \leq \langle x \otimes b \rangle.
\end{align*}
This completes the proof.
\end{proof}

\begin{lemma}\label{L_8808_Dom}
Let $A$ be a simple C*-algebra which is not of type~I.
Then for every $m \in \N$,
every $b \in (A^{\otimes m})_{+} \setminus \{ 0 \}$,
every $x \in A_{+}$ with $\| x \| = 1$,
and for every $\ep > 0$,
there are $\dt > 0$ and $z \in A_{+} \setminus \{ 0 \}$
such that, whenever $a \in A_{+}$ satisfies $\| a \| \leq 1$
and $( x^2 - x a x - \dt )_{+} \precsim z$,
then
$\big[ (x^2)^{\otimes m} - (x a x)^{\otimes m} - \ep \bigr]_{+}
  \precsim b$.
\end{lemma}

\begin{proof}
First suppose that there is $y \in A_{+} \setminus \{ 0 \}$
such that $b = y^{\otimes m}$.
Set
\[
\dt = \frac{\ep}{4 m^2}.
\]
By \cite[Lemma 1.13]{Ph14},
there are $n \in \N$ and $c_1, c_2, \ldots, c_n \in A$
such that $\left\| x - \sum_{j = 1}^n c_j^* y c_j \right\| < \dt$.
Then, as in the proof of Lemma~\ref{L_8825_DFSub},
\begin{equation}\label{Eq_8817_NN}
(x - \dt)_{+}
  \precsim 1_n \otimes y,
\end{equation}
whence $ 1_n \otimes y$ in $M_n (A)$.
Therefore there is $\dt_0 > 0$ such that
\begin{equation}\label{Eq_8808_Label}
(x - 2 \dt)_{+} \precsim (1_n \otimes y - \dt_0)_{+}.
\end{equation}
Use \cite[Lemma~2.4]{Ph14} to find orthogonal
Cuntz equivalent elements
$r_1, r_2, \ldots, r_n \in ({\overline{y A y}} )_{+} \setminus \{ 0 \}$.
Set $r = r_1$.

Apply Lemma~\ref{L_8808_Switch} with $y$ in place of both $a$ and~$b$,
with $r$ in place of~$x$,
and with $\dt_0$ in place of~$\ep$,
and call the resulting element~$s_0$.
Choose $s \in A_{+} \setminus \{ 0 \}$
such that $s \precsim s_0$ and $s \precsim r$.
Then
\begin{equation}\label{Eq_8817_Sw}
(y - \dt_0)_{+} \otimes s \precsim r \otimes y.
\end{equation}

Use \cite[Lemma~2.4]{Ph14} to find orthogonal
Cuntz equivalent elements
$z_{j_1, j_2, \ldots, j_{m - 1}}
 \in ({\overline{s A s}} )_{+} \setminus \{ 0 \}$
for $j_1, j_2, \ldots, j_{m - 1} = 1, 2, \ldots, n$.
Set $z = z_{1, 1, \ldots, 1}$.

Suppose $( x^2 - x a x - \dt )_{+} \precsim z$.
For $k = 1, 2, \ldots, m$ define
\[
h_k
 = \big[ (x - 2 \dt)_{+} a (x - 2 \dt)_{+} \bigr]^{\otimes (k - 1)}
  \otimes ( x^2 - x a x - \dt )_{+}
  \otimes
     \big[ \bigl( (x - 2 \dt)_{+} \bigr)^2 \bigr]^{\otimes (m - k )}
\]
and
\[
h_k^{(0)}
 = (x a x)^{\otimes (k - 1)}
  \otimes ( x^2 - x a x )
  \otimes ( x^2 )^{\otimes (m - k )}.
\]
Using
$\| x \| \leq 1$ and $\| a \| \leq 1$,
we get
$\bigl\| h_k -  h_k^{(0)} \bigr\| \leq (4 m - 3) \dt$.
Since
\[
\sum_{k = 1}^m h_k^{(0)} = ( x^2 )^{\otimes m} - (x a x)^{\otimes m},
\]
it follows that
\begin{equation}\label{Eq_8808_LabelTwo}
\Biggl\| ( x^2 )^{\otimes m}
    - (x a x)^{\otimes m} - \sum_{k = 1}^m h_k \Biggr\|
 \leq m (4 m - 3) \dt
 < \ep.
\end{equation}

We have, using (\ref{Eq_8817_NN}) at the second step,
\begin{align*}
h_1
& = ( x^2 - x a x - \dt )_{+}
  \otimes \big[ \bigl( (x - 2 \dt)_{+} \bigr)^2 \bigr]^{\otimes (m - 1 )}
\\
& \precsim z \otimes [1_n \otimes y]^{\otimes (m - 1 )}
\\
& \sim \sum_{j_1, j_2, \ldots, j_{m - 1} = 1}^n
   z_{j_1, j_2, \ldots, j_{m - 1}} \otimes y^{\otimes (m - 1 )}
\\
& \precsim s \otimes y^{\otimes (m - 1 )}
  \precsim r \otimes y^{\otimes (m - 1 )}
  \sim r_1 \otimes y^{\otimes (m - 1 )}.
\end{align*}

For $k = 2, 3, \ldots, m$,
using (\ref{Eq_8808_Label})
and (\ref{Eq_8817_NN}) at the third step
and using (\ref{Eq_8817_Sw}) at the sixth step,
\begin{align*}
h_k
& = \big[ (x - 2 \dt)_{+} a (x - 2 \dt)_{+} \bigr]^{\otimes (k - 1)}
  \otimes ( x^2 - x a x - \dt )_{+}
  \otimes
   \big[ \bigl( (x - 2 \dt)_{+} \bigr)^2 \bigr]^{\otimes (m - k )}
\\
& \precsim
   \big[ \bigl( (x - 2 \dt)_{+} \bigr)^2 \bigr]^{\otimes (k - 1)}
  \otimes ( x^2 - x a x - \dt )_{+}
  \otimes
    \big[ \bigl( (x - 2 \dt)_{+} \bigr)^2 \bigr]^{\otimes (m - k )}
\\
& \precsim \big[ 1_n \otimes (y - \dt_0)_{+} \bigr]
    \otimes [1_n \otimes y]^{\otimes (k - 2 )}
    \otimes z
    \otimes [1_n \otimes y]^{\otimes (m - k )}
\\
& \sim \sum_{j_1, j_2, \ldots, j_{m - 1} = 1}^n
   (y - \dt_0)_{+}
   \otimes y^{\otimes (k - 2 )}
   \otimes z_{j_1, j_2, \ldots, j_{m - 1}} \otimes y^{\otimes (m - k )}
\\
& \precsim (y - \dt_0)_{+}
   \otimes y^{\otimes (k - 2 )}
   \otimes s \otimes y^{\otimes (m - k )}
\\
& \precsim r
   \otimes y^{\otimes (k - 2 )}
   \otimes y \otimes y^{\otimes (m - k )}
  = r \otimes y^{\otimes (m - 1 )}
  \sim r_k \otimes y^{\otimes (m - 1 )}.
\end{align*}
Therefore
\[
\sum_{k = 1}^m h_k
 \precsim
    \Biggl( \sum_{k = 1}^m r_k \Biggr) \otimes y^{\otimes (m - 1 )}
 \precsim y^{\otimes m}.
\]
It now follows from~\eqref{Eq_8808_LabelTwo}
that
\[
\big[ (x^2)^{\otimes m} - (x a x)^{\otimes m} - \ep \bigr]_{+}
  \precsim \sum_{k = 1}^m h_k
  \precsim y^{\otimes m}.
\]
This completes the proof of the special case.

For the general case,
use induction and Kirchberg's Slice Lemma (\cite[Lemma~4.1.9]{Ro02})
to find $y_1, y_2, \ldots, y_m \in A_{+} \setminus \{ 0 \}$
such that $y_1 \otimes y_2 \otimes \cdots \otimes y_m \precsim b$.
Use \cite[Lemma 2.6]{Ph14}
to choose $y \in A_{+} \setminus \{ 0 \}$
such that $y \precsim y_k$ for $k = 1, 2, \ldots, m$.
Then $y^{\otimes m} \precsim b$.
The choices of $\dt$ and $z$ which work for~$y$
in the special case also work for~$b$.
\end{proof}

\begin{lemma}\label{L_8809_UseAppId}
Let $A$ be a simple C*-algebra, let $G$ be a finite group,
and let $\alpha \colon G \to \mathrm{Aut} (A)$
be an action of $G$ on~$A$.
Let $x \in A_{+} \setminus \{ 0 \}$.
Suppose that an element $y \in A_{+}$ has the following property.
For any finite set $F \subseteq A$,
any $\varepsilon > 0$,
and any $x \in A_{+} \setminus \{ 0 \},$
there exist orthogonal positive contractions
$f_g \in A$ for $g \in G$ such that,
with $f = \sum_{g \in G} f_{g}$,
the following hold:
\begin{enumerate}
\item\label{L_8809_UseAppId_WTRP-COM}
$\| f_{g} a - a f_{g} \| < \varepsilon$
for all $a \in F$ and all $g \in G$.
\item\label{L_8809_UseAppId_WTRP-LTR}
$\|\alpha_{g} (f_{h}) - f_{g h} \| < \varepsilon$ for all $g, h \in G$.
\item\label{L_8809_UseAppId_WTRP-CU-EQ}
$\bigl( y^{2} - y f y - \varepsilon \bigr)_{+} \precsim x$.
\end{enumerate}
Then every positive element $z \in \overline{A y}$
also has the same property.
\end{lemma}

\begin{proof}
The proof is essentially the same as that of \cite[Lemma~3.5]{AGJP21}.
\end{proof}

\begin{lemma}\label{L_8814_TRP_Fin}
Let $A$ be a finite simple C*-algebra, let $G$ be a finite group,
and let $\alpha \colon G \to \mathrm{Aut} (A)$
be an action of $G$ on~$A$.
Suppose that for every $\varepsilon > 0$,
every finite set $F \subseteq A$,
and every $x, y \in A_{+}$ with $x \neq 0$,
there exist orthogonal positive contractions
$f_g \in A$ for $g \in G$ such that,
with $f = \sum_{g \in G} f_{g}$,
the following hold:
\begin{enumerate}
\item\label{L_8814_TRP_Fin_WTRP_COM}
$\| f_{g} a - a f_{g} \| < \varepsilon$
for all $a \in F$ and all $g \in G$.
\item\label{L_8814_TRP_FinWTRP_LTR}
$\|\alpha_{g} (f_{h}) - f_{g h} \| < \varepsilon$ for all $g, h \in G$.
\item\label{L_8814_TRP_FinWTRP_CU_EQ}
$\bigl( y^{2} - y f y - \varepsilon \bigr)_{+} \precsim x$.
\end{enumerate}
Then $\alpha$ has the weak tracial Rokhlin property.
\end{lemma}

The conditions are the same as in
Definition~\ref{defwtrp},
except that we have omitted~(\ref{WTRP-OUT-CO}),
the requirement that
$\|f x f\| > 1 - \varepsilon$.

\begin{proof}[Proof of Lemma~\ref{L_8814_TRP_Fin}]
The proof is the same as that of \cite[Proposition~7.12]{AGJP21}.
\end{proof}

\begin{theorem}\label{T_8816_PermwTRP}
Let $A$
be a simple tracially $\mathcal{Z}$-absorbing C*-algebra,
let $m \in \N$,
and adopt Notation~\ref{N_8808_Setup}.
Suppose that $A^{\otimes m}$ is finite
(\cite[Definition~7.1]{AGJP21}).
Then the permutation action $\bt \colon S_m \to \Aut (A^{\otimes m})$
has the weak tracial Rokhlin property.
\end{theorem}

\begin{proof}
We will make repeated use of the following estimate.
Suppose
\[
a_1, a_2, \ldots, a_m, b_1, b_2, \ldots, b_m \in A
\]
satisfy $\| a_j \| \leq 1$ and $\| b_j \| \leq 1$
for $j = 1, 2, \ldots, m$.
Then
\begin{equation}\label{Eq_8816_TensEst}
\bigl\| a_1 \otimes a_2 \otimes \dots \otimes a_m
  - b_1 \otimes b_2 \otimes \dots \otimes b_m \bigr\|
 \leq \sum_{j = 1}^m \| a_j - b_j \|.
\end{equation}

We verify the conditions in Lemma~\ref{L_8814_TRP_Fin}.
Thus, let $\varepsilon > 0$,
let $F \subseteq A^{\otimes m}$ be finite,
and let $x, y \in (A^{\otimes m})_{+}$ with $x \neq 0$.
We need to find
orthogonal positive contractions
$f_{\sm} \in A^{\otimes m}$ for $\sm \in S_m$ such that,
with $f = \sum_{\sm \in S_m} f_{\sm}$,
the following hold:
\begin{enumerate}
\item\label{T_8816_PermwTRP_COM}
$\| f_{\sm} a - a f_{\sm} \| < \varepsilon$
for all $a \in F$ and all $\sm \in S_m$.
\item\label{T_8816_PermwTRP_LTR}
$\| \bt_{\sm} (f_{\ta}) - f_{\sm \ta} \| < \varepsilon$
for all $\sm, \ta \in S_m$.
\item\label{T_8816_PermwTRP_CU_EQ}
$\bigl( y^{2} - y f y - \varepsilon \bigr)_{+} \precsim x$.
\setcounter{TmpEnumi}{\value{enumi}}
\end{enumerate}
It is clearly sufficient
(changing~$\ep$)
to prove this for $y$
in a dense subset of $(A^{\otimes m})_{+}$.
Combining this fact with Lemma~\ref{L_8809_UseAppId},
we see that we need only consider elements~$y$
in an approximate identity for $A^{\otimes m}$.
Thus,
we may assume that there is $y_0 \in A$ such that
\begin{equation}\label{Eq_8816_y0}
y_0 \geq 0,
\qquad
\| y_0 \| = 1,
\andeqn
y = y_0^{\otimes m}.
\end{equation}
We may further assume that there is a finite set $F_0 \subseteq A$
such that $\| a \| \leq 1$ for all $a \in F_0$ and
\begin{equation}\label{Eq_8816_DefF0}
F = \bigl\{ a_1 \otimes a_2 \otimes \dots \otimes a_m \colon
 a_1, a_2, \ldots, a_m \in F_0 \bigr\}.
\end{equation}

Use \cite[Lemma 2.4]{Ph14}
to choose nonzero orthogonal elements
$b, c \in \bigl( {\overline{x (A^{\otimes m}) x }} \bigr)_{+}$.

Use induction and Kirchberg's Slice Lemma (\cite[Lemma~4.1.9]{Ro02})
to find $t_1, t_2, \ldots, t_m \in A_{+} \setminus \{ 0 \}$
such that $t_1 \otimes t_2 \otimes \cdots \otimes t_m \precsim c$.
Use \cite[Lemma 2.6]{Ph14}
to choose $c_0 \in A_{+} \setminus \{ 0 \}$
such that $c_0 \precsim t_k$ for $k = 1, 2, \ldots, m$.
Then $c_0^{\otimes m} \precsim c$.
Without loss of generality,
$\| c_0 \| = 1$.

Define
\[
\ep_0 = \min \left( \frac{1}{2}, \, \frac{\ep}{32 m} \right)
\andeqn
y_1 = (y_0 - \ep_0)_{+}.
\]
Apply Lemma~\ref{L_8825_DFSub}
with $y_0$ in place of~$a$,
with $c_0$ in place of~$b$,
and with $\ep_0$ in place of~$\ep$,
getting $n_0 \in \N$.

We have
\[
\lim_{t \to \infty}
 \left[ n_0^m - \left( n_0 - \frac{m}{t} \right)^m \right]
= 0.
\]
Therefore there is $n_1 \in \N$ with $n_1 >m$ such that
\begin{equation}\label{Eq_8816_n1Est}
n_0^m - \left( n_0 - \frac{m}{n_1} \right)^m < 1.
\end{equation}
Define
\begin{equation}\label{Eq_8816_nDef}
n = n_1 n_0.
\end{equation}

Define
\[
\dt = \min \left( \frac{\ep_0}{n + 1}, \, \frac{\ep}{16 m n^m} \right).
\]
Apply Lemma~\ref{L_8808_CommEqEst}
with $\dt$ in place of~$\ep$,
getting a number $\rh_1 > 0$
(called $\dt$ there).

Use \cite[Lemma 2.5]{AP16} to choose $\rh_2 > 0$
such that whenever
$s, t \in A_{+}$
satisfy
\[
\| s \| \leq 1,
\qquad
\| t \| \leq 1,
\qquad {\mbox{and}} \qquad
\| s t - t s \| < \rh_2,
\]
then
$\bigl\| s^{1/2} t - t s^{1/2} \bigr\| < \dt$.
Apply Lemma~\ref{L_8808_Dom} with $b$ and~$m$ as given,
with $\frac{\ep}{2}$ in place of~$\ep$,
and with $y_0$ in place of~$x$,
getting $b_0  \in A_{+} \setminus \{ 0 \}$ (called $z$ there)
and $\rh_3 > 0$ (called $\dt$ there).
Set
\[
\rh = \min \left( \rh_1, \, \rh_2, \, \rh_3, \, \frac{\ep}{n^m} \right).
\]
\cite[Definition~3.6]{AGJP21}, now provides
a c.p.c.~order zero map $\varphi \colon M_{n} \to A$
such that:
\begin{enumerate}
\setcounter{enumi}{\value{TmpEnumi}}
\item\label{T_8816_PermwTRP_deftza_it1}
$\bigl( y_0^{2} - y_0 \varphi (1) y_0 - \rh \bigr)_{+}
 \precsim b_0$.
\item\label{T_8816_PermwTRP_deftza_it2}
$\| [\varphi (z), a] \| < \rh$
for any $z \in M_{n}$ with $\| z \| \leq 1$
and any $a \in F_0 \cup \{ y_0 \}$.
\end{enumerate}

Define $N = \{ 1, 2, \ldots, n \}$.
Define
\[
R = \bigl\{ r = \bigl( r (1), \, r (2), \, \ldots, \, r (m) \bigr)
   \in N^m \colon r (1) < r (2) < \cdots < r (m) \bigr\}.
\]
The group $S_m$ acts on $N^m$
by $(\sm \cdot r) (j) = r (\sm^{-1} (j))$
for $\sm \in S_m$,
$r = \bigl( r (1), \, r (2), \, \ldots, \, r (m) \bigr) \in N^m$,
and $j = 1, 2, \ldots, m$.
Using this action,
we see that $S_m \cdot R$
is the set of all $r \in N^m$ such that the numbers
$r (1), \, r (2), \, \ldots, \, r (m)$ are all distinct.
Define $Q = N^m \setminus S_m \cdot R$.
We have
\begin{equation}\label{Eq_8816_CardSmR}
\card (S_m \cdot R)
 = n (n - 1) (n - 2) \cdots (n - m + 1)
 \geq (n - m)^m,
\end{equation}
so
\begin{equation}\label{Eq_8816_CardQ}
\card (Q) \leq n^m -  (n - m)^m.
\end{equation}

To simplify notation,
for $k \in N$ define $g_k = \ph (e_{k, k})$,
and for $r \in N^m$ define
\[
g_r = g_{r (1)} \otimes g_{r (2)} \otimes \cdots \otimes g_{r (m)}.
\]
One checks that $\bt_{\sm} (g_r) = g_{\sm \cdot r}$
for $\sm \in S_m$ and $r \in N^m$.
Also, if $\sigma\cdot r= \sigma'\cdot r'$
for some $\sigma,\sigma' \in S_m$ and $r,r' \in R$, then
$\sigma =\sigma' $ and  $r=r'$.

For $\sm \in S_m$,
define
\[
f_{\sm} = \sum_{r \in R} g_{\sm \cdot r}.
\]
Since $\ph$ is a c.p.c.~order zero map,
$f_{\sm}$ is a positive contraction.
The sets $\sm \cdot R$ are disjoint,
so the elements $f_{\sm}$ are orthogonal.
Moreover,
for $\sm, \ta \in S_m$,
\[
\bt_{\sm} (f_{\ta})
 = \sum_{r \in R} \bt_{\sm} (g_{\ta \cdot r})
 = \sum_{r \in R} g_{\sm \cdot \ta \cdot r}
 = f_{\sm \ta}.
\]
So (\ref{T_8816_PermwTRP_LTR}) holds.

We check~(\ref{T_8816_PermwTRP_COM}).
Let $a_1, a_2, \ldots, a_m \in F_0$ and set
$a = a_1 \otimes a_2 \otimes \dots \otimes a_m$.
Let $r \in N^m$.
Then one checks that
\[
[ g_r, a ]
 = \sum_{j = 1}^m a_1 g_{r (1)} \otimes
       \cdots \otimes a_{j - 1} g_{r (j - 1)}
       \otimes [ g_{r (j)}, a_j] \otimes g_{r (j + 1)} a_{j + 1}
       \otimes
       \cdots \otimes g_{r (m)} a_m.
\]
Since $\| a_j \| \leq 1$ and $\| g_{r (j)} \| \leq 1$
for $j = 1, 2, \ldots, m$,
we have
$\| [ g_r, a ] \| < m \rh$ by~(\ref{T_8816_PermwTRP_deftza_it2}).
Therefore, for $\sm \in S_m$,
\[
\| [ f_{\sm}, a ] \|
  \leq \sum_{r \in \sm \cdot R} \| [ g_{r}, a ] \|
  < \card (R) m \rh
  \leq n^m \rh
  \leq \ep.
\]
This is~(\ref{T_8816_PermwTRP_COM}).

It remains to prove~(\ref{T_8816_PermwTRP_CU_EQ}),
which requires considerable work.
First,
the relations $\rh \leq \rh_3$
and
$\bigl( y_0^{2} - y_0 \varphi (1) y_0 - \rh \bigr)_{+}
 \precsim b_0$,
together with the choice of $b_0$ and $\rh_3$
using Lemma~\ref{L_8808_Dom},
imply that
\begin{equation}\label{Eq_8816_yytm}
\left( y^2 - y \ph (1)^{\otimes m} y - \tfrac{\ep}{2} \right)
  \precsim b.
\end{equation}

Next, for future reference,
at the second step use $\rh \leq \rh_2$,
the choice of $\rh_2$,
and~(\ref{T_8816_PermwTRP_deftza_it2}) to get, for any $k=1,\ldots,n$,
\begin{equation}\label{Eq_8825_FEst}
\bigl\| y_0 g_k y_0 - g_k^{1/2} y_0^2 g_k^{1/2} \bigr\|
 \leq 2 \bigl\| \bigl[ y_0, g_k^{1/2} \bigr] \bigr\|
 < 2 \dt.
\end{equation}

Applying~(\ref{T_8816_PermwTRP_deftza_it2}) to~$y_0$,
the inequality $\rh \leq \rh_1$
and the choice of $\rh_1$ using Lemma~\ref{L_8808_CommEqEst}
provide, for $k = 1, 2, \ldots, n$,
elements $v_k, w_k \in A$ such that
$\| v_k \|, \, \| w_k \| \leq 1$,
\begin{equation}\label{Eq_8817_vkwk}
\bigl\| v_k^* y_0 g_k y_0 v_k - y_0 g_1 y_0 \bigr\| < \dt,
\quad {\mbox{and}} \quad
\bigl\| w_k^* y_0 g_1 y_0 w_k - y_0 g_k y_0 \bigr\| < \dt.
\end{equation}

We claim that
\begin{equation}\label{Eq_8817_Sub1}
n_0 n_1 \langle ( y_0 g_1 y_0 - 4 \ep_0 )_{+} \rangle
 \leq \langle y_1 \rangle.
\end{equation}
To prove the claim,
use~(\ref{Eq_8825_FEst}) at the first step
and $\ph (1) = \sum_{k = 1}^n g_k$ at the second step to get
\begin{align*}
\biggl\| y_1 \ph (1) y_1
   - \sum_{k = 1}^n g_k^{1/2} y_0^2 g_k^{1/2} \biggr\|
& < \biggl\| y_1 \ph (1) y_1 - \sum_{k = 1}^n y_0 g_k y_0 \biggr\|
      + 2 n \dt
\\
& \leq 2 \| y_1 - y_0 \| + 2 n \dt
  \leq 2 \ep_0 + 2 n \dt.
\end{align*}
Since $g_1, g_2, \ldots, g_n$ are orthogonal,
it follows that
\begin{equation}\label{Eq_8817_SubSum}
\sum_{k = 1}^n \bigl\langle \bigl( g_k^{1/2} y_0^2 g_k^{1/2}
          - [2 \ep_0 + 2 n \dt] \bigr)_{+} \bigr\rangle
   \leq \langle y_1 \ph (1) y_1 \rangle
   \leq \langle y_1 \rangle.
\end{equation}
For $k = 1, 2, \ldots, n$,
using $(2 n + 1) \dt < 2 \ep_0$ at the first step,
using \cite[Lemma~1.4(6)]{Ph14} at the second step,
using (\ref{Eq_8817_vkwk}), $\| v_k \| \leq 1$,
and \cite[Lemma~2.3]{AGJP21}, at the third step,
we get
\begin{align*}
( y_0 g_1 y_0 - 4 \ep_0 )_{+}
& \precsim \bigl( y_0 g_1 y_0 - [(2 n + 1) \dt + 2 \ep_0] \bigr)_{+}
\\
& \sim \bigl( g_1^{1/2} y_0^2 g_1^{1/2}
     - [(2 n + 1) \dt + 2 \ep_0] \bigr)_{+}
\\
& \precsim \bigl( g_k^{1/2} y_0^2 g_k^{1/2}
     - [2 n \dt + 2 \ep_0] \bigr)_{+}.
\end{align*}
Summing over~$k$ and combining this with (\ref{Eq_8817_SubSum}) gives
$n \langle ( y_0 g_1 y_0 - 4 \ep_0 )_{+} \rangle
 \leq \langle y_1 \rangle$,
so the claim follows from~(\ref{Eq_8816_nDef}).

We next claim that
\begin{equation}\label{Eq_8817_cBound}
\bigl[ n^m - (n - m)^m \bigr]
 \bigl\langle \bigl[
     ( y_0 g_1 y_0 - 4 \ep_0 )_{+} \bigr]^{\otimes m} \bigr\rangle
 \leq \langle c \rangle.
\end{equation}
To prove this,
use (\ref{Eq_8817_Sub1})
and the choice of $n_0$ using Lemma~\ref{L_8825_DFSub}
to get
$n_1 \langle ( y_0 g_1 y_0 - 4 \ep_0 )_{+} \rangle
 \leq \langle c_0 \rangle$.
So
\[
n_1^m \bigl\langle \bigl[
     ( y_0 g_1 y_0 - 4 \ep_0 )_{+} \bigr]^{\otimes m} \bigr\rangle
 \leq \langle c_0^{\otimes m} \rangle
 \leq \langle c \rangle.
\]
The claim now follows by multiplying (\ref{Eq_8816_n1Est})
by $n_1^m$ and using (\ref{Eq_8816_nDef}) again.

Now we claim that for every $r \in N^m$, we have
\begin{equation}\label{Eq_8817_TPDom}
\left( g_{r (1)}^{1/2} y_0^2 g_{r (1)}^{1/2}
     \otimes g_{r (2)}^{1/2} y_0^2 g_{r (2)}^{1/2}
     \otimes \cdots
     \otimes g_{r (m)}^{1/2} y_0^2 g_{r (m)}^{1/2}
   - \frac{\ep}{4} \right)_{+}
  \precsim \bigl[ ( y_0 g_1 y_0 - 4 \ep_0 )_{+} \bigr]^{\otimes m}.
\end{equation}
We prove the claim.
Recalling the elements $w_k$ in~(\ref{Eq_8817_vkwk}),
for $r \in N^m$ define
\[
w_r = w_{r (1)} \otimes w_{r (2)} \otimes \cdots \otimes w_{r (m)}
 \in A^{\otimes m}.
\]
Then,
using (\ref{Eq_8816_TensEst}) at the first step
and (\ref{Eq_8825_FEst}) and \eqref{Eq_8817_vkwk} at the second step,
\begin{align*}
& \Bigl\| w_r^* \bigl[ ( y_0 g_1 y_0 - 4 \ep_0 )_{+} \bigr]^{\otimes m}
                                    w_r
\\
& \hspace*{6em} {\mbox{}}
        -  g_{r (1)}^{1/2} y_0^2 g_{r (1)}^{1/2}
     \otimes g_{r (2)}^{1/2} y_0^2 g_{r (2)}^{1/2}
     \otimes \cdots
     \otimes g_{r (m)}^{1/2} y_0^2 g_{r (m)}^{1/2} \Bigr\|
\\
& \hspace*{3em} {\mbox{}}
  \leq \sum_{j = 1}^m \Bigl[ \| w_{r (j)}^* \|
       \bigl\| ( y_0 g_1 y_0 - 4 \ep_0 )_{+} -y_0 g_1 y_0 \bigr\|
                  \| w_{r (j)} \|
\\
& \hspace*{6em} {\mbox{}}
          + \bigl\| w_{r (j)}^* y_0 g_1 y_0 w_{r (j)}
                            - y_0 g_{r (j)} y_0 \bigr\|
          + \bigl\| y_0 g_{r (j)} y_0
                - g_{r (j)}^{1/2} y_0^2 g_{r (j)}^{1/2} \bigr\| \Bigr]
\\
& \hspace*{3em} {\mbox{}}
  < m ( 4 \ep_0 + \dt + 2 \dt )
  = 4 m \ep_0 + 3 m \dt
  \leq \frac{\ep}{8} + \frac{\ep}{8}
  = \frac{\ep}{4}.
\end{align*}
The claim follows.

We now have,
using (\ref{Eq_8816_TensEst}) at the second step
and (\ref{Eq_8825_FEst}) at the third step,
%
\begin{align*}
&
\Biggl\| \sum_{r \in Q} g_{r (1)}^{1/2} y_0^2 g_{r (1)}^{1/2}
     \otimes g_{r (2)}^{1/2} y_0^2 g_{r (2)}^{1/2}
     \otimes \cdots
     \otimes g_{r (m)}^{1/2} y_0^2 g_{r (m)}^{1/2}
   - \sum_{r \in Q} y_0^{\otimes m} g_r y_0^{\otimes m} \Biggr\|
\\
& \hspace*{3em} {\mbox{}}
\leq \sum_{r \in Q} \Bigl\| g_{r (1)}^{1/2} y_0^2 g_{r (1)}^{1/2}
     \otimes g_{r (2)}^{1/2} y_0^2 g_{r (2)}^{1/2}
     \otimes \cdots
     \otimes g_{r (m)}^{1/2} y_0^2 g_{r (m)}^{1/2}
   - y_0^{\otimes m} g_r y_0^{\otimes m} \Bigr\|
\\
& \hspace*{3em} {\mbox{}}
\leq \sum_{r \in Q} \sum_{j = 1}^m
  \bigl\| y_0 g_{r (j)} y_0
                - g_{r (j)}^{1/2} y_0^2 g_{r (j)}^{1/2} \bigr\|
< 2 \card (Q) m \dt
\leq 2 m n^m \dt
\leq \frac{\ep}{4}.
\end{align*}

Therefore, using \cite[Corollary~1.6]{Ph14}
at the first step
and orthogonality of $g_1, g_2, \ldots, g_n$
at the second step,
\begin{align*}
& \Biggl( \sum_{r \in Q} y_0^{\otimes m} g_r y_0^{\otimes m}
   - \frac{\ep}{2} \Biggr)_{+}
\\
& \hspace*{3em} {\mbox{}}
  \precsim
  \Biggl( \sum_{r \in Q} g_{r (1)}^{1/2} y_0^2 g_{r (1)}^{1/2}
     \otimes g_{r (2)}^{1/2} y_0^2 g_{r (2)}^{1/2}
     \otimes \cdots
     \otimes g_{r (m)}^{1/2} y_0^2 g_{r (m)}^{1/2}
   - \frac{\ep}{4} \Biggr)_{+}
\\
& \hspace*{3em} {\mbox{}}
  = \sum_{r \in Q} \Bigl( g_{r (1)}^{1/2} y_0^2 g_{r (1)}^{1/2}
     \otimes g_{r (2)}^{1/2} y_0^2 g_{r (2)}^{1/2}
     \otimes \cdots
     \otimes g_{r (m)}^{1/2} y_0^2 g_{r (m)}^{1/2}
   - \frac{\ep}{4} \Bigr)_{+}.
\end{align*}
So,
using (\ref{Eq_8817_TPDom})
at the first step
and (\ref{Eq_8817_cBound}) and (\ref{Eq_8816_CardQ})
at the second step,
\[
\left\langle \Biggl( \sum_{r \in Q} y_0^{\otimes m} g_r y_0^{\otimes m}
   - \frac{\ep}{2} \Biggr)_{+} \right\rangle
\leq \card (Q) \left\langle
   \bigl[ ( y_0 g_1 y_0 - 4 \ep_0 )_{+} \bigr]^{\otimes m}
          \right\rangle
\leq \langle c \rangle.
\]

Finally, use
\cite[Lemma 1.5]{Ph14} to combine
this last inequality with
(\ref{Eq_8816_yytm})
and
$y \ph (1) y - f = \sum_{r \in Q} y_0^{\otimes m} g_r y_0^{\otimes m}$
to get
\begin{align*}
( y^2 - y f y - \ep )_{+}
& \precsim
   \left( y^2 - y \ph (1)^{\otimes m} y - \tfrac{\ep}{2} \right)
   \oplus \Biggl( \sum_{r \in Q} y_0^{\otimes m} g_r y_0^{\otimes m}
     - \frac{\ep}{2} \Biggr)_{+}
\\
&\precsim b \oplus c
 \precsim x.
\end{align*}
This is~(\ref{T_8816_PermwTRP_CU_EQ}),
and finishes the proof.
\end{proof}

The following result follows from Theorems~\ref{T_8816_PermwTRP}
and \ref{thm_fg}, and Corollary~\ref{Cor_intermediate}.

\begin{corollary}\label{cor_per_tz_fin}
Let $A$ be a simple  
tracially $\mathcal{Z}$-absorbing C*-algebra  and let
 $m\in\mathbb{N}$. Suppose that $A^{\otimes m}$ is finite
 in the sense of \cite[Definition~7.1]{AGJP21}. Then
 the crossed product and the fixed point algebra
are simple and tracially $\mathcal{Z}$-absorbing.
 If, in addition, $A$ is $\sigma$-unital,
 then all intermediate C*-algebras of the inclusions
$(A^{\otimes m})^\beta \subseteq A^{\otimes m}$ and 
$A^{\otimes m} \subseteq C^*(S_m, A^{\otimes m},\beta)$
are simple and tracially $\mathcal{Z}$-absorbing. 
\end{corollary}

The following result follows from the fact that every
finite group $G$ embeds into some permutation group
$S_m$, Theorem~\ref{T_8816_PermwTRP}, and \cite[Proposition~4.1]{FG17}.

\begin{corollary}\label{cor_per_any}
For any finite group $G$ and any finite, simple,
self-absorbing, and tracially $\mathcal{Z}$-absorbing C*-algebra $A$,
there is an action  $\alpha \colon G \to \mathrm{Aut} (A)$ with
the weak tracial Rokhlin property.
\end{corollary}

\begin{proposition}\label{prop_outer}
Let $A$ be a simple  
purely infinite C*-algebra and let
 $m\in\mathbb{N}$. Then
 the permutation action $\bt \colon S_m \to \Aut (A^{\otimes m})$
 is pointwise outer and the crossed product and the fixed
 point algebra are simple and purely infinite.
 If, in addition, $A$ is $\sigma$-unital,
 then all intermediate C*-algebras of the inclusions
$(A^{\otimes m})^\beta \subseteq A^{\otimes m}$ and 
$A^{\otimes m} \subseteq C^*(S_m, A^{\otimes m},\beta)$
are simple and purely infinite. 
\end{proposition}

\begin{proof}
Suppose that there is $\sigma\in S_m\setminus \{1\}$
and a unitary $u$ in the multiplier algebra of $A^{\otimes m}$
such that $\beta_\sigma(c)=ucu^*$ for all $c\in A^{\otimes m}$.
Let $q\in A$ be any nonzero projection and put
$p=q^{\otimes m}$ and $B=qAq$. Then $p$ is a $\beta$-invariant projection
and so $upu^*=\beta_\sigma(p)=p$. Thus $pup=pu=up$. Put
$v=pup$ which is a unitary in $B^{\otimes m}$ and
$\beta_\sigma(c)=vcv^*$ for all $c\in B^{\otimes m}$.
Now, Sakai's result \cite[Theorem~7]{Sakai1975} implies that 
$B$ is isomorphic
to some full matrix algebra $M_n$.
(Note that in the Sakai's result, it seems that
the assumption of the unitality is implicit as
$K(H)$ also satisfies the conclusion of that result.
However, his argument can be modified for the nonunital
case to cover the algebra $K(H)$.)
This is a contradiction as $B$ is purely infinite.

By \cite[Theorem~3]{Jeong1995},
pointwise outernesss of $\beta$ implies that
the crossed product $ C^*(S_m, A^{\otimes m},\beta)$
is simple and purely infinite. Since the fixed
point algebra $(A^{\otimes m})^\beta$ is isomorphic
to a full corner of $ C^*(S_m, A^{\otimes m},\beta)$,
it is is simple and purely infinite.
If $A$ is $\sigma$-unital,
the same conclusion holds for all
intermediate C*-algebras of the inclusions
$(A^{\otimes m})^\beta \subseteq A^{\otimes m}$ and 
$A^{\otimes m} \subseteq C^*(S_m, A^{\otimes m},\beta)$,
since, by \cite[Corollary~6.6]{Izumi2002}, they are, 
respectively,
 fixed point algebras and crossed products
 of the restriction of the action 
 $\bt \colon S_m \to \Aut (A^{\otimes m})$ to 
 subgroups of $S_m$.
\end{proof}

\begin{corollary}\label{cor_per_tz}
Let $A$ be a simple separable exact tracially $\mathcal{Z}$-absorbing
C*-algebra. Let $m\in\mathbb{N}$ and consider
 the permutation action $\bt \colon S_m \to \Aut (A^{\otimes m})$.
 Then  all intermediate C*-algebras of the inclusions
$(A^{\otimes m})^\beta \subseteq A^{\otimes m}$ and 
$A^{\otimes m} \subseteq C^*(S_m, A^{\otimes m},\beta)$
are simple and tracially $\mathcal{Z}$-absorbing. 
If, in addition, $A$ is nuclear then 
all these intermediate C*-algebras 
 are   nuclear and $\mathcal{Z}$-stable.
\end{corollary}

\begin{proof}
Since $A$ is simple, separable, and  tracially $\mathcal{Z}$-absorbing,
either $A$ is purely infinite or $sr(A)=1$, by 
\cite[Theorem~4.11 and Corollary~6.5]{FLL21}.
If $A$ is purely infinite, the first part of the
 statement follows from
Proposition~\ref{prop_outer}. In the sequel, suppose that
$A$ has stable rank one.

Since $A$ is tracially $\mathcal{Z}$-absorbing,
so is $A^{\otimes m}$, by \cite[Theorem~5.1]{AGJP21}.
Then, again by \cite[Theorem~4.11 and Corollary~6.5]{FLL21},
either $A^{\otimes m}$ is purely infinite or $sr(A^{\otimes m})=1$.
If  $sr(A^{\otimes m})=1$, then $A^{\otimes m}$ is finite
in the sense of \cite[Definition~7.1]{AGJP21}
(that is, its unitization is finite), and hence
 the first part of the statement follows from 
 Corollary~\ref{cor_per_tz_fin}.
 
 To complete the proof of the first part of the statement,
it is enough to show that if  $A$ has stable rank one,
then $A^{\otimes m}$ cannot be purely infinite.
Recall that a simple C*-algebra (unital or not) admits a densely 
defined quasitrace  if and only if no matrix algebra over it 
contains an infinite projection \cite{BlkCnt1982, BlkHdm1982}.
Since $A$ is simple 
with $sr(A)=1$, this implies that $A$ has a densely defined 
quasitrace $\tau$, which is actually
a densely defined trace, by the exactness of $A$. 
Now, $\tau^{\otimes m}$ is a densely defined
 trace on $A^{\otimes m}$ (whose domain is 
 the algebraic tensor product of $m$ copies of $A$). 
 So, $A^{\otimes m}$ does not have any infinite
projection, and hence it is not purely infinite.

 For the second part of the statement,
 the nuclearity of the intermediate C*-algebras follows from 
Izumi's   Galois correspondence for
  intermediate C*-algebras 
  \cite[Corollary~6.6]{Izumi2002}, and their
   $\mathcal{Z}$-stability follows from \cite[Theorem~A]{CLS2021}.
\end{proof}

\section{Integer actions}\label{sec_int}

\indent
In this section we show---under a mild additional assumption---that
(nonunital)
tracial $\mathcal{Z}$-absorption
passes to crossed products by automorphisms
with the weak tracial Rokhlin property,
as in Definition~\ref{defwtrpaut} below.
(See Theorem~\ref{thm_int} below.)
This result is the nonunital version
of \cite[Theorem 6.7]{HO13}.
We weaken
the assumption in \cite[Theorem 6.7]{HO13}
that $\alpha^{m}$ acts trivially on $T (A)$ for some
$m \in \mathbb{N}$,
and we don't need separability.

We point out that
the proof of \cite[Theorem 6.7]{HO13} isn't valid without exactness,
or at least without assuming that
every quasitrace is a trace.
Strict comparison is not defined in \cite{HO13},
so it isn't clear whether the intention there is
to use traces or quasitraces.
The version using quasitraces is needed
for the part of \cite[Theorem 3.3]{HO13}
that says ``and therefore $A$ has strict comparison'',
and version using traces is needed
at the end of the third last paragraph of
the proof of \cite[Lemma 6.6]{HO13}.

First we extend the definition of the (weak) tracial Rokhlin property
for actions of
$\mathbb{Z}$ to the nonunital case (cf.~\cite[Definition~1.1]{OP06}).
The analogous definition in~\cite{HO13}, Definition~6.1 there,
asks for orthogonal positive contractions
$e_{1}, e_{2}, \ldots, e_{n}$,
but we use the more conventional indexing
$e_{0}, e_{1}, \ldots, e_{n}$.

\begin{definition}\label{defwtrpaut}
Let $A$ be a simple C*-algebra and let $\alpha \in \mathrm{Aut} (A)$.
We say that $\alpha$ has the \emph{weak tracial Rokhlin property}
if for every finite set $F \subseteq A$, every $\varepsilon > 0$,
every $n \in \mathbb{N}$,
and every $x, y \in A_{+}$ with $\| x \| = 1$,
there exist orthogonal positive contractions
$e_{0}, e_{1} \ldots, e_{n}$ in $A$ such that,
with $e = \sum_{j = 0}^{n} e_{j}$, the following hold:
\begin{enumerate}
\item\label{defwtrpaut_it1}
$\| \alpha (e_{j}) - e_{j + 1} \| < \varepsilon$
for $j = 0, 1, \ldots, n - 1$.
\item\label{defwtrpaut_it2}
$\| [e_{j}, b] \| \leq \varepsilon$
for $j = 0, 1, \ldots, n$ and all $b \in F$.
\item\label{defwtrpaut_it3}
$\big( y^{2} - y e y - \varepsilon \big)_{+} \precsim x$.
\item\label{defwtrpaut_it4}
$\|e x e\| > 1 - \varepsilon$.
\end{enumerate}
\end{definition}

\begin{proposition}\label{P_8615_WeakZToGenZ}
Let $A$ be a simple unital C*-algebra
and let $\alpha \in \mathrm{Aut} (A)$.
\begin{enumerate}[label=$\mathrm{(\arabic*)}$]
\item\label{P_8615_WeakZToGenZ_ToGen}
If $\alpha$ has the weak tracial Rokhlin property,
then $\alpha$ has the generalized tracial Rokhlin property
of \cite[Definition 6.1]{HO13}.
\item\label{P_8615_WeakZToGenZ_ToWk}
If $A$ is finite
and $\alpha$ has the generalized tracial Rokhlin property
of \cite[Definition 6.1]{HO13},
then $\alpha$ has the weak tracial Rokhlin property.
\end{enumerate}
\end{proposition}

\begin{proof}
The proof is similar to
that of Proposition~\ref{P_8615_WeakToGen},
using $\{ 1, 2, \ldots, n \}$ in place of~$G$.

For~\ref{P_8615_WeakZToGenZ_ToGen},
for given $\ep > 0$,
$n \in \N$,
$F \subseteq A$ finite,
and $a \in A_{+} \setminus \{ 0 \}$,
we will ask for
orthogonal positive contractions $e_{0}, e_{1}, \ldots, e_{n} \in A$
satisfying the conditions of \cite[Definition 6.1]{HO13},
rather than $e_{1}, e_{2}, \ldots, e_{n}$;
since $n$ is arbitrary,
this is equivalent.

Choose $\dt_1, \dt_2 > 0$ as in the proof of
Proposition \ref{P_8615_WeakToGen}\ref{P_8615_WeakToGen_ToGen}.
Set
\[
\ep_0 = \min \left( \dt_1, \, \dt_2, \, \frac{1}{2 (n + 1)} \right).
\]
Define $x$ and~$y$ as
in the proof of
Proposition \ref{P_8615_WeakToGen}(\ref{P_8615_WeakToGen_ToGen}).
Apply Definition~\ref{defwtrpaut}
with $\varepsilon_0$ in place of $\varepsilon$ and
with $x,y,F,n$ as given,
getting
$f_0, f_1, \ldots, f_n$ (called
$e_0, e_1, \ldots, e_n$ in Definition~\ref{defwtrpaut}),
and define $f = \sum_{j = 0}^{n} f_{j}$.
Then
\[
\| f \|
 > \sqrt{1 - \ep_0}
 \geq \sqrt{1 - \frac{1}{2 (n + 1)}}
 \geq 1 - \frac{1}{2 (n + 1)}.
\]
Therefore there is $j_0 \in \{ 0, 1, \ldots, n \}$
such that
\[
\| f_{j_0} \| > 1 - \frac{1}{2 (n + 1)}.
\]
Using Definition~\ref{defwtrpaut}\eqref{defwtrpaut_it1}
and induction,
for $j = 0, 1, \ldots, n$ we get
\[
\| f_j \| > 1 - \frac{1}{2 (n + 1)} - \frac{| j - j_0 |}{2 (n + 1)}.
\]
In particular,
$\| f_j \| > \frac{1}{2} > \frac{1}{3}$ for all~$j$.
{}From here,
finish as in the proof of
Proposition \ref{P_8615_WeakToGen}\ref{P_8615_WeakToGen_ToGen}.

The proof of~\ref{P_8615_WeakZToGenZ_ToWk}
is the same as the proof of
Proposition \ref{P_8615_WeakToGen}\ref{P_8615_WeakToGen_ToWk}.
\end{proof}

\begin{proposition}\label{outer}
Let $A$ be a nonzero simple C*-algebra
and let $\alpha \in \mathrm{Aut} (A)$
have the weak tracial Rokhlin property.
Then $\alpha^{n}$ is outer
for every $n \in \mathbb{Z} \setminus \{ 0 \}$.
\end{proposition}

\begin{proof}
It is enough to prove the case $n > 0$.
So assume $n \in \mathbb{Z}$ and $n > 0$.
Let $u \in M (A)$ be unitary, and assume that
$\af^n = {\mathrm{Ad}} (u)$;
we derive a contradiction.

Define
\[
\dt = \frac{1}{n^2 + 3 n + 9}.
\]
Use \cite[Lemma~2.5]{AP16} to choose
$\delta_0 > 0$ such that whenever $D$ is a C*-algebra
and $h, k \in D$ satisfy $0 \leq h, k \leq 1$
and $\| [h, k]\| < \delta_0$, then $\| [h^{1/2}, k] \| < \delta$.
We also require $\delta_0 \leq \delta$.
Choose $b_0 \in A_{+}$ such that $\| b_0 \| = 1$.
Set
\[
F = \bigl\{ \af^{k} (b_0)
       \colon k = - n, \, - n + 1, \, \ldots, \, n \bigr\}
     \cup \bigl\{ u \af^{- k} \big( b_0^{1/2} \big)
       \colon k = 0, 1, \ldots, n \bigr\}.
\]

Apply Definition~\ref{defwtrpaut} with this choice of~$F$,
with $\delta_0$ in place of $\varepsilon$,
with $x = b_0$, and with $y = 0$.
We get orthogonal positive contractions
$f_{0}, f_{1}, \ldots, f_{n} \in A$ such that,
with $f = \sum_{j = 1}^{n} f_{j}$, the following hold:
\begin{enumerate}
\item\label{WTRP-O-F-1a}
$\big\| f_{j} \af^{k} (b_0) - \af^{k} (b_0) f_{j} \big\| < \delta_0$
for all $j \in \{0, 1, \ldots, n \}$
and all $k \in \{ - n, \, - n + 1, \, \ldots, \, n \}$.
\item\label{WTRP-O-F-1b}
$\big\| f_{j} u \af^{- k} \big( b_0^{1/2} \big)
        - u \af^{- k} \big( b_0^{1/2} \big) f_{j} \big\|
   < \delta_0$
for all $j, k \in \{0, 1, \ldots, n \}$.
\item\label{WTRP-O-F-2}
$\|\alpha (f_{j}) - f_{j + 1} \| < \delta_0$
for all $j \in \{0, 1, \ldots, n - 1 \}$.
\item\label{WTRP-O-F-3}
$\| f b_0 f \| > 1 - \delta_0$.
\setcounter{TmpEnumi}{\value{enumi}}
\end{enumerate}
It follows that
(using the choice of $\dt_0$
for (\ref{WTRP-O-F-C5})
and~(\ref{WTRP-O-F-C_New03})):
\begin{enumerate}
\setcounter{enumi}{\value{TmpEnumi}}
\item\label{WTRP-O-F-C4}
Whenever $j, k \in \{0, 1, \ldots, n \}$
then $\| \af^{j - k} (f_k) - f_j \| < n \dt_0$.
\item\label{WTRP-O-F-C5}
$\big\| f_{j}^{1/2} \af^{k} (b_0)
     - \af^{k} (b_0) f_{j}^{1/2} \big\| < \delta$
for all $j \in \{0, 1, \ldots, n \}$
and all $k \in \{ - n, \, - n + 1, \, \ldots, \, n \}$.
\item\label{WTRP-O-F-C_New03}
$\big\| f_{j} \af^{k} (b_0)^{1/2}
     - \af^{k} (b_0)^{1/2} f_{j} \big\| < \delta$
for all $j \in \{0, 1, \ldots, n \}$
and all $k \in \{ - n, \, - n + 1, \, \ldots, \, n \}$.
\end{enumerate}

We claim that there exists $l \in \{0, 1, \ldots, n \}$
such that
\[
\| f_l b_0 f_l \| > 1 - (n^2 + n + 1) \dt.
\]
To prove the claim,
first use \eqref{WTRP-O-F-1a}
and orthogonality of $f_{0}, f_{1}, \ldots, f_{n}$
to see that if $j, k \in \{0, 1, \ldots, n \}$ with $j \neq k$,
then $\| f_j b_0 f_k \| < \dt_0 \leq \dt$,
and
\[
\Bigg\| \sum_{j = 0}^{n} f_j b_0 f_j \Bigg\|
 = \max_{0 \leq j \leq n} \| f_j b_0 f_j \|.
\]
Using these two facts at the third step,
we get
\[
1 - \dt
 < \| f b_0 f \|
 \leq \Bigg\| \sum_{j = 0}^{n} f_j b_0 f_j \Bigg\|
      + \sum_{j \neq k} \| f_j b_0 f_k \|
 < \max_{0 \leq j \leq n} \| f_j b_0 f_j \| + n (n + 1) \dt.
\]
The claim follows.

Now let $l$ be as in the claim,
and define $b = \af^{- l} (b_0)$.
Then, using the claim and (\ref{WTRP-O-F-C4}) at the third step,
\begin{align}\label{Eq_8613_f0b}
\| f_0 b \|
& = \| \af^l (f_0) b_0 \|
  \geq \| f_l b_0 \| - \| \af^l (f_0) - f_l \|\cdot \| b_0 \|
\\
& > \bigl[ 1 - (n^2 + n + 1) \dt) \bigr] - n \dt
  = 1 - (n^2 + 2 n + 1) \dt.
\notag
\end{align}
%
Define $c = f_0^{1/2} \af^n (b) f_0^{1/2}$
and $d = f_n^{1/2} \af^n (b) f_n^{1/2}$.
Using (\ref{WTRP-O-F-C5}) at the first and last steps,
using (\ref{WTRP-O-F-C4}) at the second step,
using (\ref{WTRP-O-F-C_New03}) at the fourth step,
and using (\ref{WTRP-O-F-1b}) at the fifth step,
we get
\begin{align}\label{Eq_8613_MainApp}
d
& \approx_{\dt} f_n \af^n (b)
  \approx_{n \dt} \af^n (f_0 b)
  = u f_0 b u^*
\\
& \approx_{\dt} u b^{1/2} f_0 b^{1/2} u^*
  \approx_{\dt_0} f_0 u b^{1/2} b^{1/2} u^*
  = f_0 \af^n (b)
  \approx_{\dt} c.
\notag
\end{align}
Using~(\ref{Eq_8613_f0b})
and the appearance $u f_0 b u^*$
in the middle of~(\ref{Eq_8613_MainApp}),
we get
\[
\| c \|
   > 1 - (n^2 + 2 n + 1) \dt - (\dt_0 + 2 \dt)
   \geq 1 - (n^2 + 2 n + 4) \dt.
\]
Also from~(\ref{Eq_8613_MainApp}),
we get
\[
\| c - d \|
   < (n + 3) \dt + \dt_0
   \leq (n + 4) \dt.
\]
Therefore,
at the second step using the fact that
$c$ and $d$ are orthogonal positive elements,
\[
(n + 4) \dt
 > \| c - d \|
 = \max (\| c \|, \, \| d \|)
 > 1 - (n^2 + 2 n + 4) \dt,
\]
which implies
$(n^2 + 3 n + 8) \dt > 1$.
This inequality contradicts the choice of~$\dt$.
\end{proof}

\begin{corollary}\label{C_8613_CrPrdSimple}
Let $A$ be a simple C*-algebra
and let $\alpha \in \mathrm{Aut} (A)$
have the weak tracial Rokhlin property.
Then $C^{*} (\mathbb{Z}, A, \alpha)$ is simple.
\end{corollary}

\begin{proof}
Combine Proposition~\ref{outer}
and \cite[Theorem~3.1]{Ki81}.
\end{proof}

The condition in the following definition substitutes for
the hypothesis in \cite[Lemma 6.6 and Theorem 6.7]{HO13}
that some power of the automorphism acts trivially
on the tracial state space.
The difference from Definition~\ref{defwtrpaut}
is the addition of the last condition.

\begin{definition}\label{defwtrpaut_cont}
Let $A$ be a simple C*-algebra and let $\alpha \in \mathrm{Aut} (A)$.
We say that $\alpha$ has
the \emph{controlled weak tracial Rokhlin property}
if for every $z \in A_{+} \setminus \{ 0 \}$
there exists $n_{0} \in \mathbb{N}$ such that
for every $n \in \mathbb{N}$ with $n \geq n_{0}$,
every finite set $F \subseteq A$, every $\varepsilon > 0$,
and every $x, y  \in A_{+}$ with $\| x \| = 1$,
there exist orthogonal positive contractions
$e_{0}, e_{1}, \ldots, e_{n}$ in $A$ such that,
with $e = \sum_{j = 0}^{n} e_{j}$, the following hold:
\begin{enumerate}
\item\label{defwtrpautc_it1}
$\| \alpha (e_{j}) - e_{j + 1} \| < \varepsilon$
for $j = 0, 1, \ldots, n - 1$.
\item\label{defwtrpautc_it2}
$\| [e_{j}, b] \| < \varepsilon$
for $j = 0, 1, \ldots, n$ and all $b \in F$.
\item\label{defwtrpautc_it3}
$(y^{2} - y e y - \varepsilon)_{+} \precsim x$.
\item\label{defwtrpautc_it4}
$\|e x e\| > 1 - \varepsilon$.
\item\label{defwtrpautc_it5}
$e_{j} \precsim z$ for $j = 0, 1, \ldots, n$.
\end{enumerate}
\end{definition}

We do not know any example
of an automorphism of a simple C*-algebra which has
the weak tracial Rokhlin property
but not the controlled weak tracial Rokhlin property,
although we suspect such examples exist.
It seems less likely
that such examples can exist
on a simple tracially $\mathcal{Z}$-absorbing C*-algebra.

Note that if $\alpha \in \mathrm{Aut} (A)$ has 
the controlled weak tracial Rokhlin property then it
has the  weak tracial Rokhlin property, since by
taking a fixed $z_0\in A_+\setminus \{0\}$,
Definition~\ref{defwtrpaut_cont} implies that there
is $n_0\in\mathbb{N}$ such that Definition~\ref{defwtrpaut}
holds for any $n\geq n_0$. Now if we are given 
$F,\varepsilon,n,x,y$ as in Definition~\ref{defwtrpaut},
then we can apply Definition~\ref{defwtrpaut} 
 with $m=(n+1)n_0 -1$ in place
of $n$ and with $\varepsilon/n_0$ in place of $\varepsilon$
to get $e_0,\ldots,e_m$. Set $f_j=\sum_{k=0}^{n_0 -1} e_{j+k(n+1)}$,
for $0\leq j\leq n$. Then $f_j$'s satisfy Definition~\ref{defwtrpaut}.

\begin{remark}\label{rmk_cwtrp}
Definition~\ref{defwtrpaut_cont}
is equivalent to Definition~\ref{defwtrpaut} if $A$ is purely infinite,
since in this case, $a\precsim z$ for all $a,z\in A_+\setminus \{0\}$.
\end{remark}

\begin{proposition}\label{P_8616_GenToWeakControl}
Let $A$ be a  simple unital exact
tracially $\mathcal{Z}$-absorbing C*-algebra
and let $\alpha \in \mathrm{Aut} (A)$
have the generalized tracial Rokhlin property
of \cite[Definition 6.1]{HO13}.
Suppose that there is $m > 0$ such that $\af^m$
acts trivially on the tracial state space $T (A)$.
Then $\alpha$ has the controlled weak tracial Rokhlin property.
\end{proposition}

\begin{proof}
If $A$ is traceless then by \cite[Corllary~5.1]{Ro04}, $A$ is purely infinite.
Then by Remark~\ref{rmk_cwtrp} we are done. In the
sequel, assume that $T(A)\neq \emptyset$.

We first prove the version without
the condition
$\| e x e \| > 1 - \ep$;
see (\ref{Cond_8703_Norm1}),
(\ref{Cond_8703_Trans}),
(\ref{Cond_8703_Comm}),
(\ref{Cond_8703_Small}),
and (\ref{Cond_8817_6}) below.

Let $z \in A_{+} \setminus \{ 0 \}$.
Set $\rho = \inf_{\ta \in T (A)} d_{\ta} (z)$.
Choose some $r \in \big( {\overline{z A z}} \big)_{+}$
such that $\| r \| = 1$.
Then $\rh \geq \inf_{\ta \in T (A)} \ta (r) > 0$.

By \cite[Lemma~2.4]{Ph14},
for every $k \in \N$
there is $b \in A_{+} \setminus \{ 0 \}$
such that $k \langle b \rangle \leq \langle 1 \rangle$
in $W (A)$,
and by \cite[Lemma~2.4]{Ph14}
for every $x \in A_{+} \setminus \{ 0 \}$
there is $c \in A_{+} \setminus \{ 0 \}$
such that $c \precsim x$ and $c \precsim z$.
Using these two facts and \cite[Lemma~6.5]{HO13},
we see that there is $n_0 \in \N$ such that
for every $n \in \mathbb{N}$ with $n \geq n_0$,
every finite set $F \subseteq A$, every $\varepsilon > 0$,
and every $x \in A_{+} \setminus \{ 0 \}$,
there exist orthogonal positive contractions
$e_{0}, e_{1}, \ldots, e_{n}$ in $A$ such that,
with $e = \sum_{j = 0}^{n} e_{j}$, the following hold:
\begin{enumerate}
\item\label{Cond_8703_Norm1}
$\| e_j \| = 1$ for $j = 0, 1, \ldots, n$.
\item\label{Cond_8703_Trans}
$\| \alpha (e_{j}) - e_{j + 1} \| < \varepsilon$
for $j = 0, 1, \ldots, n - 1$.
\item\label{Cond_8703_Comm}
$\| [e_{j}, b] \| \leq \varepsilon$
for $j = 0, 1, \ldots, n$ and all $b \in F$.
\item\label{Cond_8703_Small}
$1 - \sum_{j = 0}^n e_j \precsim x$.
\item\label{Cond_8703_Controlled}
$d_{\ta} (e_j) < \rho$
for $j = 0, 1, \ldots, n$.
\setcounter{TmpEnumi}{\value{enumi}}
\end{enumerate}
Since by \cite[Theorem~3.3]{HO13} $A$ has strict comparison
and, by \cite[Theorem 5.11]{Haa}, all quasitraces on~$A$ are traces, so
the last condition implies that:
\begin{enumerate}
\setcounter{enumi}{\value{TmpEnumi}}
\item\label{Cond_8817_6}
$e_j \precsim z$
for $j = 0, 1, \ldots, n$.
\end{enumerate}

The rest of the proof
is now the same as the proof of
Proposition~\ref{P_8615_WeakToGen}\ref{P_8615_WeakToGen_ToWk}.
\end{proof}

Combing Proposition~\ref{P_8615_WeakToGen}\ref{P_8615_WeakToGen_ToGen},
 the previous propostion, and \cite[Remark~6.8]{HO13}, we obtain 
 the following:

\begin{corollary}\label{cor_cwtrp}
Let $A$ be a  simple unital exact
tracially $\mathcal{Z}$-absorbing C*-algebra
such that $T(A)$ has finitely many extremal traces.
If $\alpha \in \mathrm{Aut} (A)$
has the weak tracial Rokhlin property,
then it has the controlled weak tracial Rokhlin property.
\end{corollary}

\begin{lemma}\label{L_8612_AppInv}
Let $A$ be a stably finite C*-algebra,
let $G$ be a discrete group,
and let $\alpha \colon G \to {\mathrm{Aut}} (A)$
be an action of $G$ on~$A$.
Let $F \subseteq A$
and $S \subseteq G$ be finite.
Then for every $\varepsilon > 0$,
there is $e \in A$
such that $0 \leq e \leq 1$,
$\| e a - a \| < \varepsilon$
for all $a \in F$,
and $\| \alpha_g (e) - e \| < \varepsilon$
for all $g \in S$.
\end{lemma}

The group $G$ need not be amenable,
and there are no conditions at all on the action.
The methods are taken from the proof of \cite[Theorem 4.6]{AP16}.
The situation there,
however, is much more complicated,
and it seems simplest to give a full proof here,
relying on several lemmas from \cite{AP16}.

\begin{proof}[Proof of Lemma~\ref{L_8612_AppInv}]
Choose an integer $n\geq 2$ such that $n > \frac{4}{\varepsilon}$.
Define subsets of~$G$ by
\[
S_0 = \{ 1 \},
\qquad
S_1 = S \cup S^{-1} \cup \{ 1 \},
\qquad
S_2 = S_1 S_1,
\]
\[
S_3 = S_1 S_2,
\qquad
\ldots,
\qquad
S_n = S_1 S_{n - 1}.
\]

Use
\cite[Lemma 4.4]{AP16}
to choose
$\rh > 0$ such that
whenever $e, x \in A$ satisfy
\[
0 \leq e \leq 1,
\qquad
0 \leq x \leq 1,
\andeqn
\| e x - x \| < \rh,
\]
then $\big\| e^{1/2} x - x \big\| < \frac{\varepsilon}{4}$.

For $k = 0, 1, \ldots, n$,
let $D_k \subseteq A$ be the hereditary
subalgebra of~$A$ generated by all $\alpha_g (a)$
for $g \in S_k$ and $a \in F$.
Then $\alpha_g (D_k) \subseteq D_{k + 1}$
for $g \in S_1$ and $k = 0, 1, \ldots, n - 1$.
Choose $c_0 \in D_0$
such that $0 \leq c_0 \leq 1$ and
\begin{equation}\label{Eq_8613_AppId0}
\| c_0 a - a \| < \rh
\end{equation}
for $a \in F$.

By induction, choose
$c_k \in D_k$ for $k = 1, 2, \ldots, n$
such that $0 \leq c_k \leq 1$,
\begin{equation}\label{Eq_8613_AppIdk}
\| c_k \alpha_g (c_l) - \alpha_g (c_l) \| < \rh
\end{equation}
for
\[
k = 1, 2, \ldots, n,
\qquad
g \in S_1,
\andeqn
l = 0, 1, \ldots, k - 1,
\]
and
\begin{equation}\label{Eq_8613_AppIdEst}
\| c_k a - a \| < \varepsilon
\end{equation}
for $k = 0, 1, \ldots, n$
and $a \in F$.

Now define
$e = \frac{1}{n} \sum_{k = 0}^{n - 1} c_k$.
Clearly $0 \leq e \leq 1$.

For $a \in F$ we use~(\ref{Eq_8613_AppIdEst})
to get
\[
\| e a - a \|
 \leq \frac{1}{n} \sum_{k = 0}^{n - 1} \| c_k a - a \|
 < \left( \frac{1}{n} \right) n \varepsilon
 = \varepsilon.
\]

It remains to prove that $\| \alpha_g (e) - e \| < \varepsilon$
for all $g \in S$.
So let $g \in S$.
We will apply \cite[Lemma 4.3]{AP16}.
We define $c_{-1} = 0$,
and take the selfadjoint elements $a, b, d, r, x, y \in A$
to be
\[
a = \frac{1}{n} \sum_{k = 0}^{n - 2} c_k,
\qquad
b = \frac{1}{n} \sum_{k = 0}^{n} c_k,
\qquad
d = e,
\qquad
r = \alpha_g (e),
\]
\[
x = \frac{1}{n} \sum_{k = 0}^{n - 1}
          \alpha_g ( c_k^{1/2} )  c_{k - 1} \alpha_g ( c_k^{1/2} ),
\andeqn
y = \frac{1}{n} \sum_{k = 0}^{n - 1}
          c_{k + 1}^{1/2} \alpha_g ( c_k ) c_{k + 1}^{1/2}.
\]
To apply  \cite[Lemma 4.3]{AP16}, in place of 
 $\varepsilon$ we use $\frac{2}{n}$,
and in place of $\rh$ we use~$\frac{\varepsilon}{2}$.

We verify the hypotheses of \cite[Lemma 4.3]{AP16}.
That $a \leq d \leq b$ is clear.
The relations $x \leq r$ and $y \leq b$
follow from
\[
\alpha_g ( c_k^{1/2} ) c_{k - 1} \alpha_g ( c_k^{1/2} )
  \leq \alpha_g ( c_k )
\andeqn
c_{k + 1}^{1/2} \alpha_g ( c_k ) c_{k + 1}^{1/2}
 \leq c_{k + 1}
\]
for $k = 0, 1, \ldots, n - 1$.
We also have
\[
\| b - a \|
 = \left\| \frac{1}{n} \sum_{k = 0}^{n} c_k
            - \frac{1}{n} \sum_{k = 0}^{n - 2} c_k \right\|
 \leq \frac{1}{n} \big( \| c_{n - 1} \| + \| c_n \| \big)
 \leq \frac{2}{n}.
\]

It remain to prove that
$\| a - x \| < \frac{\varepsilon}{2}$
and
$\| r - y \| < \frac{\varepsilon}{2}$.

The choice of $\rh$
and the containment $S^{-1} \subseteq S_1$
imply that for $k = 0, 1, \ldots, n - 1$ we have
\[
\big\| c_{k + 1}^{1/2} \alpha_g ( c_k ) - \alpha_g ( c_k ) \big\|
 < \frac{\varepsilon}{4},
\qquad
\big\| \alpha_g ( c_k ) c_{k + 1}^{1/2} - \alpha_g ( c_k ) \big\|
 < \frac{\varepsilon}{4},
\]
\[
\big\| \alpha_g ( c_{k + 1}^{1/2} ) c_{k} - c_{k} \big\|
 < \frac{\varepsilon}{4},
\andeqn
\big\| c_{k} \alpha_g ( c_{k + 1}^{1/2} ) - c_{k} \big\|
 < \frac{\varepsilon}{4}.
\]
These relations imply that
\[
\| a - x \|
 < \left( \frac{1}{n} \right) (n - 1)
       \cdot 2 \left( \frac{\varepsilon}{4} \right)
 < \frac{\varepsilon}{2}
\quad {\mbox{and}} \quad
\| r - y \|
 < \left( \frac{1}{n} \right) n
       \cdot 2 \left( \frac{\varepsilon}{4} \right)
 = \frac{\varepsilon}{2}.
\]
The conclusion of \cite[Lemma 4.3]{AP16}
now tells us that
\[
\| \alpha_g (e) - e \|
 \leq \frac{2}{n} + \frac{\varepsilon}{2}
 < \frac{\varepsilon}{2} + \frac{\varepsilon}{2}
 = \varepsilon,
\]
as desired.
\end{proof}

The next lemma is the analog of
\cite[Lemma 6.6]{HO13}.
There are enough additional wrinkles
that we give the proof in full.

\begin{lemma}\label{lem_int}
Let $A$ be a simple tracially $\mathcal{Z}$-absorbing C*-algebra
and let $\alpha \in \mathrm{Aut} (A)$
have the controlled weak tracial Rokhlin property.
Then for every finite set $F \subseteq A$, every $\varepsilon > 0$,
every $a \in A_{+} \setminus \{ 0 \}$, every $n \in \mathbb{N}$,
and every positive contraction $x \in A$,
there exists a c.p.c.~order zero map $\psi \colon M_{n} \to A$
such that:
\begin{enumerate}[label=$\mathrm{(\arabic*)}$]
\item\label{lem7_9item1}
$\big( x^{2} - x \psi (1) x - \varepsilon \big)_{+} \precsim a$.
\item\label{lem7_9item2}
$\| [\psi (z), y] \| < \varepsilon$ for all $y \in F$
and all $z \in M_{n}$ with $\|z\| \leq 1$.
\item\label{lem7_9item3}
$\|\alpha (\psi (z)) - \psi (z) \| < \varepsilon$
for all $z \in M_{n}$ with $\|z\| \leq 1$.
\setcounter{TmpEnumi}{\value{enumi}}
\end{enumerate}
\end{lemma}

\begin{proof}
Let $F$, $\varepsilon$, $a$, $n$, and $x$ be as in the statement.
We may assume that $\| b \| \leq 1$ for all $b \in F$
and $\varepsilon < 1$.
Choose $M \in \mathbb{N}$ such that
$\tfrac{1}{M} < \tfrac{\varepsilon^{2}}{256}$.
By \cite[Lemma~2.1]{Ph14},
there is a nonzero positive element $a_{0} \in A$
with
\begin{equation}\label{Eq_8615_a0Sub}
a_{0} \otimes 1_{2 M + 2} \precsim a
\end{equation}
in $M_{\infty} (A)$.
Choose $n_{0} \in \mathbb{N}$ for $z = a_{0}$
according to Definition~\ref{defwtrpaut_cont}.
Choose $N \in \mathbb{N}$
such that $N > \max (2 M, n_{0} )$.
Choose $\delta > 0$ such that
\[
\dt \leq \min \left( 1, \, \frac{\ep}{12 N + 22},
          \, \frac{\ep^2}{64 N^2} \right).
\]

Use \cite[Proposition 2.5]{KW04})
to choose $\eta_0 > 0$ such that
whenever $\varphi \colon M_{n} \to A$
is a c.p.c.\  map
such that $\|\varphi (y) \varphi (z) \| < \eta_0$ for
all $y, z \in (M_{n})_{+}$ with $y z = 0$ and $\|y\|, \|z\| \leq 1$
(a c.p.c.\  $\et_0$-order zero map),
then there is a c.p.c.~order zero map
$\psi \colon M_{n} \to A$
such that $\|\varphi (z) - \psi (z) \| < \delta$
for all $z \in M_{n}$ with $\| z \| \leq 1$.

Use \cite[Lemma~2.5]{AP16} to choose
$\eta_1 > 0$ such that whenever $D$ is a C*-algebra
and $h, k \in D$ satisfy $0 \leq h, k \leq 1$
and $\| [h, k]\| < \eta_1$, then $\| [h^{1/2}, k] \| < \delta$.
Now set
$\eta
 = \min \big( \eta_0, \eta_1, \delta, \tfrac{1}{M},
      \tfrac{\varepsilon}{8} \big)$.

We claim that there is $b \in A_{+} \setminus \{ 0 \}$ such that:
\begin{enumerate}
\setcounter{enumi}{\value{TmpEnumi}}
\item\label{lem7_9item4}
$\bigoplus_{j = 0}^{N} \alpha^{j} (b) \precsim a_{0}$
in $M_{\infty} (A)$.
\setcounter{TmpEnumi}{\value{enumi}}
\end{enumerate}
To prove the claim,
since $A$ is not type~I,
we can use \cite[Lemma~2.1]{Ph14} to find
$c \in A_{+} \setminus \{ 0 \}$ such that
$c \otimes 1_{N + 1} \precsim a$ in $M_{\infty} (A)$.
Then \cite[Lemma~2.4]{Ph14} implies that there is
$b \in A_{+} \setminus \{ 0 \}$
such that $b \precsim \alpha^{- j} (c)$
for $j = 0, 1, \ldots, N$.
Thus
\[
\bigoplus_{j = 0}^{N} \alpha^{j} (b)
 \precsim \bigoplus_{j = 0}^{N} c
 = c \otimes 1_{N}
 \precsim a_{0}.
\]
The claim is proved.

Use Lemma~\ref{L_8612_AppInv}
to choose $h \in A$
such that $0 \leq h \leq 1$,
$\| h x - x \| < \eta$,
and $\| \alpha^k (h) - h \| < \eta$
for $k = 1, 2, \ldots, N$.

Since $\alpha$ has the controlled weak tracial Rokhlin property,
there are orthogonal positive contractions
$e_{0}, e_{1}, \ldots, e_{N} \in A$ such that,
with $e = e_{0} + e_{1} + \cdots + e_{N}$, the following hold:
\begin{enumerate}
\setcounter{enumi}{\value{TmpEnumi}}
\item\label{lem7_9item5}
$\| [e_{j}, y] \| < \eta$ for $j = 0, 1, \ldots, N$
and all $y \in F \cup \{h\}$.
\item\label{lem7_9item6}
$\|\alpha (e_{j}) - e_{j + 1} \| < \eta$ for $j = 0, 1, \ldots, N - 1$.
\item\label{lem7_9item7}
$(h^{2} - h e h - \eta)_{+} \precsim b$.
\item\label{defwtrpaut_it18}
$e_{j} \precsim b \precsim a_{0}$ for $j = 0, 1, \ldots, N$.
\setcounter{TmpEnumi}{\value{enumi}}
\end{enumerate}
Set
\[
E = \big\{ \alpha^{- j} (e_{k}),
 \, \alpha^{- j} \big( e_{k}^{1/2} \big), \, \alpha^{- j} (y)
\colon
 {\mbox{$0 \leq j, k \leq N$
        and $y \in F \cup \{h\}$}} \big\},
\]
which is a finite subset of~$A$.
Since $A$ is tracially $\mathcal{Z}$-absorbing there is
a c.p.c.~order zero map $\varphi \colon M_{n} \to A$
such that the following hold:
\begin{enumerate}
\setcounter{enumi}{\value{TmpEnumi}}
\item\label{lem7_9item8}
$\big( h^{2} - h \varphi (1) h - \eta \big)_{+} \precsim b$.
\item\label{lem7_9item9}
$\| [\varphi (z), y] \| < \tfrac{\eta}{N}$
for all $z \in M_{n}$ with $\|z\| \leq 1$
and all $y \in E$.
\setcounter{TmpEnumi}{\value{enumi}}
\end{enumerate}

For $j \in {\mathbb{Z}}$ define
\begin{equation*}
f_{j} =
\begin{cases}
0
& \hspace*{1em} j \leq 0
\\
(j / M) e_{j}
& \hspace*{1em} 1 \leq j \leq M - 1
\\
e_{j}
& \hspace*{1em} M \leq j \leq N - M
\\
[ (N - j) / M] e_{j}
& \hspace*{1em} N - M  + 1 \leq j \leq N - 1
\\
0
& \hspace*{1em} N \leq j.
\end{cases}
\end{equation*}
Using~(\ref{lem7_9item6}),
we get:
\begin{enumerate}
\setcounter{enumi}{\value{TmpEnumi}}
\item\label{lem7_9item11n}
$\| \af (f_j) - f_{j + 1} \| < \frac{1}{M} + \et$
for $j = 0, 1, \ldots, M - 1$
and $j = N - M, \, N - M + 1, \, \ldots, N - 1$,
and
$\| \af (f_j) - f_{j + 1} \| < \et$
for all other $j \in {\mathbb{Z}}$.
\setcounter{TmpEnumi}{\value{enumi}}
\end{enumerate}
Define a c.p.c.~map $\widetilde{\varphi} \colon M_{n} \to A$ by
\[
\widetilde{\varphi} (z)
 = \sum_{j = 0}^{N} f_{j}^{1/2} \alpha^{j} (\varphi (z)) f_{j}^{1/2}
\]
for $z \in M_n$.
By \eqref{lem7_9item9} and the definition of $f_j$,
for $j = 0, 1, \ldots, N$
and $z \in M_{n}$ with $\| z \| \leq 1$,
we have
\begin{equation}\label{Eq_8615_CommFull}
\| [ f_j, \, \af^j ( \ph (z)) ] \|
  \leq \| [ e_j, \, \af^j ( \ph (z)) ] \|
  = \| [ \af^{- j} ( e_j), \, \ph (z) ] \|
  < \frac{\et}{N}
\end{equation}
and similarly
%
\begin{equation}\label{Eq_8615_CommHalf}
\big\| \big[ f_j^{1/2}, \, \af^j ( \ph (z)) \big] \big\|
  < \frac{\et}{N}.
\end{equation}
It follows from~\eqref{Eq_8615_CommHalf}
that:
\begin{enumerate}
\setcounter{enumi}{\value{TmpEnumi}}
\item\label{lem7_9item12n}
$\big\| \widetilde{\varphi} (z)
    - \sum_{j = 1}^{N}
         f_{j} \alpha^{j} (\varphi (z)) \big\|
 < \eta$
for any $z \in M_{n}$ with $\| z \| \leq 1$.
\setcounter{TmpEnumi}{\value{enumi}}
\end{enumerate}
Together with orthogonality of the elements~$f_j$
and the fact that $\ph$ has order zero,
\eqref{Eq_8615_CommFull}~implies that
for $y, z \in (M_{n})_{+}$ with $\| y \|, \, \| z \| \leq 1$
and $y z = 0$,
we have
$\| \widetilde{\varphi} (y) \widetilde{\varphi} (z) \|
  < \eta \leq \eta_0$.
By the choice of $\eta_0$, there is a c.p.c.~order zero map
$\psi \colon M_{n} \to A$ such that:
\begin{enumerate}
\setcounter{enumi}{\value{TmpEnumi}}
\item\label{lem7_9item13n}
$\|\widetilde{ \varphi} (z) - \psi (z) \| < \delta$
for any $z \in M_{n}$ with $\| z \| \leq 1$.
\setcounter{TmpEnumi}{\value{enumi}}
\end{enumerate}
We will show that $\psi$ has
properties \ref{lem7_9item1}, \ref{lem7_9item2},
and~\ref{lem7_9item3} in the statement.

We prove \ref{lem7_9item1}.
Define
\[
w = \sum_{j = 0}^{M - 1}
       \left( 1 - \frac{j}{M} \right) e_{j}^{1/2} h^2 e_{j}^{1/2}
    + \sum_{j = N - M + 1}^{N}
       \left( 1 - \frac{N - j}{M} \right) e_{j}^{1/2} h^2 e_{j}^{1/2}.
\]
Then by \eqref{defwtrpaut_it18} and~(\ref{lem7_9item4}) we have
\begin{equation}\label{Eq_8615_wPrec}
w \precsim a_{0} \otimes 1_{2 M},
\end{equation}
and, using~\eqref{lem7_9item5}, $\et \leq \et_1$,
and the choice of $\et_1$ at the third step, we have
\begin{align}\label{Eq_8615_wEst}
\Bigg\| h e h
  - \sum_{j = 0}^{N} f_{j}^{1/2} h^{2} f_{j}^{1/2} - w \Bigg\|
& = \Bigg\| \sum_{j = 0}^{N} h e_j h
  - \sum_{j = 0}^{N} e_{j}^{1/2} h^{2} e_{j}^{1/2} \Bigg\|
\\
& \leq \sum_{j = 0}^{N} 2 \big\| \big[ e_{j}^{1/2}, h \big] \big\|
  < 2 (N + 1) \dt.
\notag
\end{align}
Furthermore,
for $j = 1, 2, \ldots, N$ we have
\begin{equation}\label{Eq_8615_hajh}
\big\| \af^j (h^2 - h \ph (1) h)
     - [ h^2 - h \af^j (\ph (1)) h ] \big\|
 \leq 4 \| \af^j (h) - h \|
 < 4 \et.
\end{equation}
Now,
using \eqref{Eq_8615_hajh} at the second step,
using \eqref{Eq_8615_wEst} at the third step,
using using~\eqref{lem7_9item5}, $\et \leq \et_1$,
the choice of $\et_1$,
and the definition of $f_{j}$ at the fourth step,
and using \eqref{lem7_9item13n} at the fifth step,
we have
\begin{align*}
& (h^{2} - h e h - \eta)_{+}
 + \sum_{j = 0}^{N} f_{j}^{1/2} \alpha^{j} \left((h^{2}
 - h \varphi (1) h - \eta)_{+} \right) f_{j}^{1/2}
\\
& \hspace*{3em} {\mbox{}}
\approx_{(N + 2) \eta}
(h^{2} - h e h)+ \sum_{j = 0}^{N} f_{j}^{1/2} \alpha^{j} (h^{2}
 - h \varphi (1) h) f_{j}^{1/2}
\\
& \hspace*{3em} {\mbox{}}
\approx_{4 (N + 1) \eta}
(h^{2} - h e h) + \sum_{j = 0}^{N} f_{j}^{1/2} h^{2} f_{j}^{1/2}
   - \sum_{j = 0}^{N}
            f_{j}^{1/2} h \alpha^{j} (\varphi (1)) h f_{j}^{1/2}
\\
& \hspace*{3em} {\mbox{}}
\approx_{2 (N + 1) \delta}
h^{2} - w
 - \sum_{j = 0}^{N} f_{j}^{1/2} h \alpha^{j} (\varphi (1)) h f_{j}^{1/2}
\\
& \hspace*{3em} {\mbox{}}
\approx_{2 (N + 1) \delta}
h^{2} - w
 - \sum_{j = 0}^{N} h f_{j}^{1/2} \alpha^{j} (\varphi (1)) f_{j}^{1/2} h
\\
& \hspace*{3em} {\mbox{}}
\approx_{\delta}
h^{2} - h \psi (1) h - w.
\end{align*}
Thus we get
\begin{align*}
&
\Bigg\| (h^{2} - h \psi (1) h)
\\
& \hspace*{4em} {\mbox{}}
 - \Bigg[ w + (h^{2} - h e h - \eta)_{+}
 + \sum_{j = 0}^{N} f_{j}^{1/2}
\alpha^{j} \left((h^{2} - h \varphi (1) h - \eta)_{+} \right)
     f_{j}^{1/2} \Bigg] \Bigg\|
\\
& \hspace*{2em} {\mbox{}}
 < (N + 2) \eta + 4 (N + 1) \eta
     + 2 (N + 1) \delta + 2 (N + 1) \delta + \delta
\\
& \hspace*{2em} {\mbox{}}
 \leq (6 N + 11) \delta
 \leq \frac{\varepsilon}{2}.
\end{align*}
Then, using \cite[Lemma~2.3]{AGJP21}, at the first step,
using \eqref{Eq_8615_wPrec}, \eqref{lem7_9item7},
and \eqref{lem7_9item8} at the third step,
using \eqref{lem7_9item4} at the fourth step,
and using \eqref{Eq_8615_a0Sub} at the fifth step, we have
\begin{align*}
\big[ x ( h^2 - h \ps (1) h ) x - \tfrac{\varepsilon}{2} \big]_{+}
&
\precsim \big( h^{2} - h \psi (1) h - \tfrac{\varepsilon}{2} \big)_{+}
\\
&
\precsim w + (h^{2} - h e h - \eta)_{+}
\\
& \hspace*{3em} {\mbox{}}
   + \sum_{j = 0}^{N} f_{j}^{1/2}
      \alpha^{j} \left((h^{2} - h \varphi (1) h - \eta)_{+} \right)
     f_{j}^{1/2}
\\
&
\precsim (a_{0} \otimes 1_{2 M}) \oplus b
 \oplus \bigoplus_{j = 1}^{N} \alpha^{j} (b)
\\
& \precsim (a_{0} \otimes 1_{2 M + 2})
\\
& \precsim  a.
\end{align*}
Using the choice of~$h$
and $\eta \leq \tfrac{\varepsilon}{8}$,
we have
\[
\big\| x ( h^2 - h \ps (1) h ) x - [x^2 - x \ps (1) x ] \big\|
 < 4 \et
 \leq \frac{\varepsilon}{2}.
\]
So \cite[Corollary 1.6]{Ph14}
implies
\[
\big( x^2 - x \ps (1) x - \varepsilon \big)_{+}
 \precsim
   \big[ x ( h^2 - h \ps (1) h ) x - \tfrac{\varepsilon}{2} \big]_{+}
 \precsim  a,
\]
as desired.

To prove \ref{lem7_9item2}, let $y \in {F}$
and let $z \in M_{n}$ satisfy $\| z \| \leq 1$.
Using \eqref{lem7_9item13n} at the first step,
using \eqref{lem7_9item12n} and $f_0 = 0$ at the second step,
using \eqref{lem7_9item9}, \eqref{lem7_9item5},
and the definition of $f_j$ at the third step,
and
using $\et \leq \dt$ at the fourth step,
we get
\begin{align*}
\|[\psi (z), y] \|
& < 2 \delta + \| [\widetilde{\varphi} (z), \, y ] \|
\\
& \leq2 \delta + 2 \eta +
  \Bigg\| \sum_{j = 1}^{N}
        \big[ f_{j} \alpha^{j} (\varphi (z)), \, y \big] \Bigg\|
\\
& < 2 \delta + 2 \eta + \eta + N \eta
  < 6 N \delta
  < \varepsilon,
\end{align*}
as desired.

Finally, we prove~\ref{lem7_9item3}.
In preparation for the estimate~\eqref{Eq_8615_SqErr} below,
we estimate the expressions
\[
\big\| (\alpha (f_{j}) - f_{j + 1}) (\alpha (f_{k}) - f_{k + 1}) \big\|
\]
for $j, k = 0, 1, \ldots, N - 1$.
If $j \neq k$,
then,
using \eqref{lem7_9item6}
and orthogonality of $e_0, e_1, \ldots, e_N$ at the fourth step,
\begin{align*}
\big\| (\alpha (f_{j}) - f_{j + 1}) (\alpha (f_{k}) - f_{k + 1}) \big\|
& =  \big\| - \alpha (f_{j}) f_{k + 1} - f_{j + 1} \alpha (f_{k}) \big\|
\\
& \leq \| \alpha (f_{j}) f_{k + 1} \| + \| f_{j + 1} \alpha (f_{k}) \|
\\
& \leq \| \alpha (e_{j}) e_{k + 1} \| + \| e_{j + 1} \alpha (e_{k}) \|
\\
& < 2 \eta
  \leq 2 \dt.
\end{align*}
If $j = k$,
we use~\eqref{lem7_9item11n}.
If $0 \leq j \leq M - 1$
or $N - M \leq j \leq N - 1$,
then
\[
\big\| \big( \alpha (f_{j}) - f_{j + 1} \big)^2 \big\|
  < \left( \frac{1}{M} + \et \right)^2
  \leq \frac{4}{M^2},
\]
while otherwise
\[
\big\| \big( \alpha (f_{j}) - f_{j + 1} \big)^2 \big\|
  < \et^2
  \leq \dt^2
  \leq \dt.
\]

Next, let $z \in M_{n}$ satisfy $\| z \| \leq 1$.
Then
\begin{align}\label{Eq_8615_SqErr}
& \Bigg\| \sum_{j = 0}^{N - 1} \alpha (f_j) \alpha^{j + 1} (\varphi (z))
       - \sum_{j = 0}^{N - 1}
           f_{j + 1} \alpha^{j + 1} (\varphi (z)) \Bigg\|^2
\\
& \hspace*{2em} {\mbox{}}
 = \Bigg\| \sum_{j, k = 0}^{N - 1} \alpha^{j + 1} (\varphi (z))^*
   (\alpha (f_{j}) - f_{j + 1}) (\alpha (f_{k}) - f_{k + 1})
      \alpha^{k + 1} (\varphi (z)) \Bigg\|
\notag
\\
& \hspace*{2em} {\mbox{}}
 \leq \sum_{j, k = 0}^{N - 1}
          \| \alpha^{j + 1} (\varphi (z)) \|
           \big\| (\alpha (f_{j}) - f_{j + 1})
                (\alpha (f_{k}) - f_{k + 1}) \big\|
           \| \alpha^{k + 1} (\varphi (z)) \|
\notag
\\
& \hspace*{2em} {\mbox{}}
 < 2 N (N - 1) \dt + 2 M \left( \frac{4}{M^2} \right)
         + (N - 2 M) \dt
\notag
\\
& \hspace*{2em} {\mbox{}}
 \leq 2 N^2 \dt + \frac{8}{M}
  \leq \frac{\ep^2}{32} + \frac{\ep^2}{32}
  = \frac{\ep^2}{16}.
\notag
\end{align}
%
Therefore
%
\[
\Bigg\|\alpha
 \Bigg(\sum_{j = 0}^{N - 1} f_{j} \alpha^{j} (\varphi (z)) \Bigg)
 - \sum_{j = 0}^{N - 1}
  f_{j + 1} \alpha^{j + 1} (\varphi (z)) \Bigg\|
 < \frac{\ep}{4}.
\]
Combining this estimate with two applications each
of \eqref{lem7_9item12n} and~\eqref{lem7_9item13n},
and using $f_0 = f_N = 1$,
we get
\[
\|\alpha (\psi (z)) - \psi (z) \|
 < 2 \delta + 2 \eta + \tfrac{\varepsilon}{4}
 \leq \tfrac{\varepsilon}{4}
    + \tfrac{\varepsilon}{4} + \tfrac{\varepsilon}{4}
 < \varepsilon.
\]
This finishes the proof.
\end{proof}

\begin{theorem}\label{thm_int}
Let $A$ be a simple tracially $\mathcal{Z}$-absorbing C*-algebra and let
$\alpha \in \mathrm{Aut} (A)$ have the controlled weak tracial
Rokhlin property.
Then $C^{*} (\mathbb{Z}, A, \alpha)$ is also a simple
tracially $\mathcal{Z}$-absorbing C*-algebra.
If moreover $A$ is separable and nuclear 
then   $C^{*} (\mathbb{Z}, A, \alpha)$ is $\mathcal{Z}$-stable.
\end{theorem}

The proof is similar to that of Theorem~\ref{thm_fg},
but requires several more steps.

\begin{proof}[Proof of Theorem~\ref{thm_int}]
Proposition~\ref{outer} and
\cite[Theorem 3.1]{Ki81}
imply that $C^{*} ({\mathbb{Z}}, A, \alpha)$ is simple.

To verify \cite[Definition~3.6]{AGJP21}, for $C^{*} ({\mathbb{Z}}, A, \alpha)$,
let $F \subseteq C^{*} ({\mathbb{Z}}, A, \alpha)$ be a finite set,
let $x, a \in C^{*} ({\mathbb{Z}}, A, \alpha)_{+}$ with $a \neq0$,
let $\varepsilon > 0$, and let $n \in \mathbb{N}$.
Without loss of generality $\ep < 1$.
The proof of \cite[Lemma~5.1]{HO13}
also works in the nonunital case,
and we can apply this generalization of it
to find $a_{0} \in A_{+} \setminus \{ 0 \}$
such that $a_{0} \precsim a$ in $C^{*} ({\mathbb{Z}}, A, \alpha)$.
Next,
let $(u_{i})_{i \in I}$ be an approximate identity for~$A$.
Then $(u_{i})_{i \in I}$ is also
an approximate identity for~$C^{*} ({\mathbb{Z}}, A, \alpha)$.
Therefore, by \cite[Remark~3.8]{AGJP21},
we may assume that $x \in A$.

Let $u \in M ( C^{*} ({\mathbb{Z}}, A, \alpha) )$
be the standard unitary
associated with the crossed product.
Choose $M \in {\mathbb{N}}$
such that for $y \in F$ there are
\[
c_{y, \, -M}, \, c_{y, \, - M + 1}, \, \ldots, \, c_{y, M} \in A
\]
satisfying
\[
\Bigg\| y - \sum_{m = - M}^{M} c_{y, m} u^m \Bigg\| < \frac{\ep}{4}.
\]
Define
\[
E = \big\{ c_{y, m} \colon
         {\mbox{$y \in F$
           and $- M \leq m \leq M$}} \big\}
\andeqn
R = \sup_{c \in E} \| c \|.
\]
Set
\[
\ep_0 = \frac{\ep}{2 (R M + 1) (2 M + 1)}.
\]
Apply Lemma~\ref{lem_int} with $E$ in place of $F$,
with $a_{0}$ in place of $a$, with $x$ as given,
and with $\varepsilon_0$ in place of $\varepsilon$.
We obtain a c.p.c.~order zero map $\psi_0 \colon M_{n} \to A$.
Let $\psi \colon A \to C^{*} ({\mathbb{Z}}, A, \alpha)$
be its composition
with the inclusion of $A$ in $C^{*} ({\mathbb{Z}}, A, \alpha)$.
We claim that $\psi$ satisfies
the conditions of \cite[Definition~3.6]{AGJP21},
for $C^{*} ({\mathbb{Z}}, A, \alpha)$.
Condition~(1) is clear.
For condition~(2), let $z \in M_n$ satisfy $\| z \| \leq 1$ and let $y \in F$.
Then for all $m \in {\mathbb{Z}}$
we have
$\| \af^m (\psi (z)) - \psi (z) \| \leq | m | \ep_0$.
Therefore
\begin{align*}
\| [ \psi (z), \, y ] \|
& \leq 2 \Bigg\| y - \sum_{m = - M}^{M} c_{y, m} u^m \Bigg\|
  + \sum_{m = - M}^{M} \| [ \psi (z), \, c_{y, m} u^m ] \|
\\
& < \frac{\ep}{2}
      + \sum_{m = - M}^{M} \| [ \psi (z), \, c_{y, m} ] \| \cdot\| u^m \|
      + \sum_{m = - M}^{M} \| c_{y, m} \|\cdot \| [ \psi (z), \, u^m ] \|
\\
& < \frac{\ep}{2} + (2 M + 1) \ep_0
      + R \sum_{m = - M}^{M} | m | \ep_0
\\
& \leq \frac{\ep}{2} + (2 M + 1) \ep_0 + R (2 M + 1) M \ep_0
   = \ep.
\end{align*}
This completes the proof first part of the statement. The second part follows from
the first part and \cite[Theorem~A]{CLS2021}.
\end{proof}

\begin{remark}\label{rmk_cont}
As discussed in the introduction to this section,
exactness,
or at least the assumption that every quasitrace is a trace,
is needed for the proof of \cite[Theorem 6.7]{HO13}
to be valid.
With this correction,
\cite[Theorem 6.7]{HO13}, at least for the stably finite case,
is a corollary of Theorem~\ref{thm_int},
by Proposition~\ref{P_8616_GenToWeakControl}.
Moreover, we don't need the separability assumption
in \cite[Theorem 6.7]{HO13}.
\end{remark}

In fact, one can say a little more.

\begin{remark}\label{cor_int_u}
Let $A$ be a simple unital
C*-algebra and let $\alpha \in \mathrm{Aut} (A)$.
Define the
{\emph{controlled generalized tracial Rokhlin property}}
by modifying \cite[Definition 6.1]{HO13}
in the same way that Definition~\ref{defwtrpaut_cont}
modifies Definition~\ref{defwtrpaut}.
If $A$ is tracially $\mathcal{Z}$-absorbing
and $\alpha$ has this property,
then  $C^{*} (\mathbb{Z}, A, \alpha)$
is also a simple tracially $\mathcal{Z}$-absorbing
C*-algebra.
To prove this,
let $\alpha \in \mathrm{Aut} (A)$
have the controlled generalized tracial Rokhlin property.
By \cite[Proposition~6.3]{HO13},
$\alpha^{m}$ is outer for all $m \in \mathbb{Z} \setminus \{ 0 \}$,
and hence $C^{*} (\mathbb{Z}, A, \alpha)$ is simple
by \cite[Theorem 3.1]{Ki81}.
Note that $\alpha$ satisfies
Conditions~\eqref{defwtrpaut_it1}, \eqref{defwtrpaut_it2},
and~\eqref{defwtrpaut_it3} in
Definition~\ref{defwtrpaut}
(but not necessarily Condition~\eqref{defwtrpaut_it4} there).
However, the proof of Lemma~\ref{lem_int} works for
$\alpha$ (with $x = 1$),
since Condition~\eqref{defwtrpaut_it4} of Definition~\ref{defwtrpaut}
is not used in the proof of that lemma.
Using  \cite[Lemma~3.3]{AGJP21},
the proof of Theorem~\ref{thm_int}
again works in this setting.
\end{remark}

The proof of the following proposition is  
similar to that of \cite[Theorem~4.5]{FG17}
except that we need Lemma~\ref{L_8612_AppInv}
instead of using the finiteness of the group.

\begin{proposition}\label{Prop_tensor}
Let $A$ and $B$ be simple C*-algebras and let
$\alpha\in\mathrm{Aut}(A)$ and $\beta\in\mathrm{Aut}(B)$.
If  $\alpha$  has the (controlled) weak tracial Rokhlin property, then
so does the automorphism $\alpha\otimes \beta$ of
$A\otimes_{\min} B$.
\end{proposition}

\begin{proof}
First suppose that 
$\alpha$  has the weak tracial Rokhlin property. We proceed
as in the proof of \cite[Theorem~4.5]{FG17} except that
we use Lemma~\ref{L_8612_AppInv}
to choose a suitable positive element $s\in B_+$ in item~(12) of the
proof of  \cite[Theorem~4.5]{FG17}.
We put $f_j=r_j \otimes s$ for $0\leq j\leq n$. The rest of that
proof works and  implies that 
$f_0,f_1,\ldots,f_n$ satisfy Definition~\ref{defwtrpaut}.
Thus
$\alpha\otimes \beta$ 
has the weak tracial Rokhlin property.

Now, suppose that $\alpha$  has the
controlled weak tracial Rokhlin property.
To verify Definition~\ref{defwtrpaut_cont} for
 $\alpha\otimes \beta$,
let $z_0\in A\otimes_{\min} B$ be positive and nonzero.
By Kirchberg's Slice Lemma (\cite[Lemma~4.1.9]{Ro02}),
there are nonzero positive elements
$z_1\in A$ and $z_2\in B$ such that 
$z_1 \otimes z_2 \precsim z_0$.
We proceed
as  the previous paragraph to obtain $s\in B_+$
satisfying item~(12) of the
proof of  \cite[Theorem~4.5]{FG17}.
Since $\lim_{\xi\to 0} (s-\xi)_+=s$, there is $\xi>0$ such that
 item~(12) of the
proof of  \cite[Theorem~4.5]{FG17} holds  for $(s-\xi)_+$
in place of $s$. By \cite[Proposition~1.14]{Ph14}, there is 
$m\in\mathbb{N}$ such that $(s-\xi)_+ \precsim 1_m \otimes z_2$.
By Proposition~\ref{outer}, $A$ is non-Type~I and
hence \cite[Lemma~2.4]{Ph14} provides a nonzero
positive element $z_3\in A$ such that $z_3 \otimes 1_m \precsim z_1$.
Choose $n_0$ as in Definition~\ref{defwtrpaut_cont} for the action
$\alpha$ with $z_3$ in place of $z$, getting $r_0,r_1,\ldots,r_n$  
for $n\geq n_0$.
We only need to
deal with Part~\eqref{defwtrpautc_it5} of Definition~\ref{defwtrpaut_cont}.
Put $f_j=r_j \otimes (s-\xi)_+$ for $0\leq j\leq n$.
Then 
\[
f_j \precsim z_3 \otimes (s-\xi)_+ \precsim z_3 \otimes
(1_m \otimes z_2)\precsim z_1 \otimes z_2 \precsim z_0.
\]
Therefore, $\alpha\otimes \beta$ 
has the controlled weak tracial Rokhlin property.
\end{proof}

\begin{example}\label{Ex_auto}
 Let $A$ be a simple $\mathcal{Z}$-stable C*-algebra and
 let $\beta$ be an
arbitrary automorphism of $A$. 
Let $\gamma$ be the bilateral tensor shift on
$\mathcal{Z}^{\otimes \infty}\cong \mathcal{Z}$ 
(see \cite{Sato10}).
 Consider the automorphism $\alpha= \beta\otimes \gamma$ on $A\cong A\otimes \mathcal{Z}$.
 Then $\alpha$ has the controlled weak tracial Rokhlin property.
  To see this, first,  by \cite[Example~6.9]{HO13} and 
 Proposition~\ref{P_8615_WeakZToGenZ}\ref{P_8615_WeakZToGenZ_ToWk}, 
 $\gamma$ has  the  weak tracial Rokhlin property.
 Since $\mathcal{Z}$ has a unique tracial state,
 Corollary~\ref{cor_cwtrp} implies that 
 $\gamma$ has  the controlled weak tracial Rokhlin property.
 Now, Proposition~\ref{Prop_tensor} implies that
 $\alpha$  has  the controlled weak tracial Rokhlin property.
 Hence the crossed product $C^{*} ({\mathbb{Z}}, A, \alpha)$
 is simple and tracially $\mathcal{Z}$-absorbing
 (by Theorem~\ref{thm_int}).
 If moreover, $A$ is separable and nuclear then  the crossed product
 is nuclear and $\mathcal{Z}$-stable
 (by \cite[Theorem~A]{CLS2021}).
\end{example}

\section*{Acknowledgment}
The second author would like to thank Mikael R{\o}rdam
for valuable discussions about the dichotomy used in
the proof of Corollary~\ref{corxC}.


\end{document}